\documentclass[11pt, a4paper,reqno]{amsart}
\allowdisplaybreaks[1]
\usepackage[latin1]{inputenc}
\usepackage{amsfonts,amsmath,amssymb, amscd,  graphicx, latexsym, color, exscale, enumerate, pstricks}
\DeclareGraphicsExtensions{.jpg,.pdf,.mps,.eps,.png}

\usepackage[all]{xy}
\usepackage{psfrag}

\newtheorem{thm}{Theorem}[section]
\newtheorem*{thm*}{Theorem}
\newtheorem{lem}[thm]{Lemma}

\newtheorem{prop}[thm]{Proposition}
\newtheorem{cor}[thm]{Corollary}

\theoremstyle{definition}
\newtheorem{rem}[thm]{Remark}
\newtheorem{rems}[thm]{Remarks}

\newcommand{\bs}{\texttt{\symbol{92}}}
\DeclareMathOperator{\tri}{trigr}
\DeclareMathOperator{\Diff}{Diff}

\newcommand{\Z}{\mathbb{Z}}

\newcommand{\R}{\mathbb{R}}

\newcommand{\D}{\mathbb{D}}
\newcommand{\A}{\mathbb{A}}

\newcommand{\ot}{\otimes}
\newcommand{\slfrac}[2]{\left.#1\middle/#2\right.}
\DeclareMathOperator{\B}{\mathcal{B}}
\DeclareMathOperator{\Pu}{\mathcal{P}}

\DeclareMathOperator{\Hom}{\mbox{Hom}}
\DeclareMathOperator{\F}{\mathcal{F}}

\DeclareMathOperator{\MCG}{MCG}

\DeclareMathOperator{\cat}{\mathcal{C}}
\DeclareMathOperator{\Fu}{\mathcal{F}}

\numberwithin{equation}{section}

\title[Categorical action of the extended affine type $A$ braid group]{Categorical action of the extended braid group of affine type $A$}

\author{Agn\`es Gadbled, Anne-Laure Thiel and Emmanuel Wagner}

\date{15/04/2015}

\begin{document}

\begin{abstract}
Using a quiver algebra of a cyclic quiver, we construct a faithful categorical action of the extended braid group of affine type $A$ on its bounded homotopy category of finitely generated projective modules. The algebra is trigraded and we identify the trigraded dimensions of the space of morphisms of this category with intersection numbers coming from the topological origin of the group.
\end{abstract}

\maketitle

\tableofcontents


\section*{Introduction}

In their seminal work, Khovanov and Seidel \cite{KhS} constructed two related (weak) categorical actions of the braid group. The first one is of algebraic nature and the second of geometric origin. They prove that both actions are faithful. The striking fact is that the proof of faithfulness relies on the topological nature of the braid group through its description as a mapping class group.\\

The constructions use various descriptions of the braid group: the finite Artin presentation, the description as a mapping class group and the description as the fundamental group of a configuration space. It turns also out that the braid group fits into the theory of Artin-Tits groups, as the Artin-Tits group of finite type $A$. When considering other Artin-Tits groups and looking which ones have all these various aspects or descriptions, one is led to look at the Artin-Tits group of finite type $B$. We construct in this paper a faithful algebraic categorical action of the Artin-Tits group of finite type $B$, seen as an extended braid group of affine type $A$. Before proceeding to a more detailed description, we expand a little bit on the  current situation about group actions on categories and the question of faithfulness.\\

There are plenty of actions of the braid groups which are now know, in particular with their connections to higher representation theory and link homologies. We mention here a couple of them in an non-exhaustive manner: Deligne \cite{Del}, Seidel-Thomas \cite{SeiTho}, Khovanov \cite{KhTang}, Stroppel \cite{Str}, Mazorchuk-Stroppel \cite{MaStr2,MaStr}, Rouquier \cite{R}, Khovanov-Rozansky \cite{KR2}, Webster \cite{Webint}, Cautis-Kamnitzer \cite{CauKam}, Lipshitz-Ozsvath-Thurston \cite{LipOzsThu}...
Many of them are known to be faithful. The constructions of Deligne and Rouquier put the braid group into the Artin-Tits group context and as such their constructions admit a immediate generalization to all finite Artin-Tits group of finite type. Categorical actions of Rouquier in type ADE are known to be faithful by a result of Brav and Thomas \cite{BrTho}. In addition Riche and Bezuriakhov \cite{Ric, BezRic} constructed a faithful action of the affine braid group on a derived category of coherent sheaves. For all we know most of the results of faithfulness rely on the result of Khovanov and Seidel.\\

Looking at categorical actions of the whole mapping class groups of an oriented surface with boundary components, the situation is drastically different, there is essentially only one such action given by Lipschitz-Ozsvath-Szabo in the context of bordered Heegaard-Floer homology and the proof of faithfulness in their setting is the same in spirit as the one of Khovanov and Seidel and as in this paper: the dimensions of the space of morphisms count certain minimal intersection numbers.\\

Let us also mention here some work generalizing Khovanov-Seidel-Thomas \cite{KhS, SeiTho} and making use of the spherical twists they introduced, in particular through cluster categories, see Smith \cite{Smi}, Grant-Marsh \cite{GraMar}, Qiu-Zhou \cite{Qiu, QiuZh}, Ishii-Ueda-Uehara \cite{IshUeUe}. In these latest papers, the authors recover faithfulness results for the (non-extended) affine type $A$ braid group.\\



The main features of the construction presented in this paper are the various gradings that we construct on the algebra under consideration. This gradings are natural in the sense that they correspond to gradings that appear when describing the Artin-Tits group as a mapping class group of a punctured surface and considering natural homological representation. Indeed under this topological description, the extended affine type $A$ braid group is simply the subgroup of the finite type $A$ braid group consisting of mapping classes fixing one chosen puncture, and as such acts on a $\mathbb{Z}^2$-covering of a punctured disk. One of the gradings corresponds to a winding number around the fixed puncture and this grading appears in a very non obvious way on the algebra (see Section \ref{action}). This specific grading allows us to obtain the main result of this paper.

\begin{thm*}
There exists a trigraded algebra $R_n$ such that
\begin{itemize}
\item The homotopy category of finitely generated trigraded projective modules over $R_n$ carries an action of the extended affine type $A$ braid group by exact endofunctors.
\item This categorical action is faithful.
\item The induced action on the Grothendieck group is a $2$-parameter homological representation of the extended affine type $A$ braid group.
\end{itemize}
\end{thm*}

The connections with the work of Qiu-Zhou and Riche-Bezuriakhov are not clear even if there seem to be some. The complete connection with the symplectic side of the picture is still under investigation. Let us also underline that a Lie theoretic interpretation is also missing. Such a description might help to make a link with the categorical action coming from Mackaay-Thiel \cite{MT} and would probably allow to prove its faithfulness.\\

The paper is organized as follows. In the first section, we review various definitions of the extended braid group of affine type $A$ and in the second, various representations of this group. In the third section, we define a (weak) categorical action of this braid group and describe in particular the algebra $R_n$ and the induced action at the level of the Grothendieck group. In the fourth section, we introduced some trigraded intersection numbers using a certain covering of the real projectivization of the tangent bundle of a punctured disk. In the last section, we prove that one recovers from some spaces of morphisms the trigraded intersection numbers and deduce the faithfulness of the categorical action.\\

\section{Braid groups}

Let $n$ be a fixed integer with $n \geq 3$.

\subsection{Braid groups by generators and relations}

The extended affine type $A$ braid group~$\hat{\B}_{\hat{A}_{n-1}}$ is generated by 
$$\sigma_1, \dots, \sigma_{n},\rho,$$ 
subject to the relations
\begin{align}
\sigma_i \sigma_ j & =  \sigma_ j \sigma_i & & \mbox{for distant} \ i,j=1,\dots, n \label{BEA1}\\
\sigma_i \sigma_ {i+1} \sigma_i& =  \sigma_ {i+1} \sigma_i \sigma_ {i+1}& & \mbox{for} \ i=1,\dots, n   \label{BEA2}\\
\rho \sigma_i \rho^{-1} & =  \sigma_{i+1} & & \mbox{for} \ i=1,\dots, n \label{BEA3}
\end{align}
where the indices have to be understood modulo $n$, e.g. $\sigma_{n+1}=\sigma_1$ by definition. We say that $i$ and $j$ are distant (resp. adjacent) if $j\not\equiv i\pm 1 \mod n$ (resp. $i\equiv j\pm 1\mod n$).

One can do without the generator $\sigma_n$ at the cost of adding the relations $\rho^n \sigma_i \rho^{-n}  = \sigma_{i}$ for all $ i=1,\dots, n-1$ or equivalently the relation $\rho \sigma_{n-1} \rho^{-1}  = \rho^{-1} \sigma_{1} \rho$. 

In particular, the center $Z \left( \hat{\B}_{\hat{A}_{n-1}} \right) $ of the extended affine type $A$ braid group is infinite cyclic generated by $\rho^n$.

This group can be depicted by diagrams on the cylinder as shown in Figure~\ref{fig:braidcyl} with the convention that a diagram drawn from bottom to top corresponds to a braid word read from right to left. The generator $\sigma_i$ consists in a crossing between the strands labelled $i$ and $i+1$ modulo $n$ while $\rho$ consists in a cyclic permutation of the points $1$ to $n$.

\begin{figure}[htbp]
  \begin{center}
  \psfrag{1}{$\scriptscriptstyle{1}$}
  \psfrag{2}{$\scriptscriptstyle{2}$}
  \psfrag{i-1}{$\scriptscriptstyle{i-1}$}
  \psfrag{i}{$\scriptscriptstyle{i}$}
  \psfrag{i+1}{$\scriptscriptstyle{i+1}$}
  \psfrag{i+2}{$\scriptscriptstyle{i+2}$}
  \psfrag{n-1}{$\scriptscriptstyle{n-1}$}
  \psfrag{n}{$\scriptscriptstyle{n}$}
   \includegraphics[height=4cm]{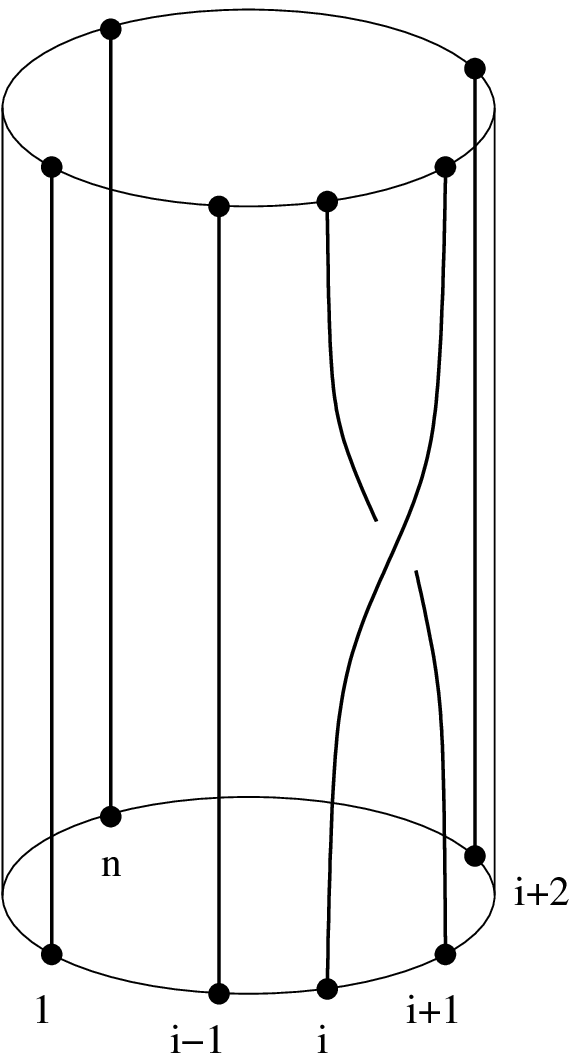} 
   \hspace{2.5cm}
   \includegraphics[height=4cm]{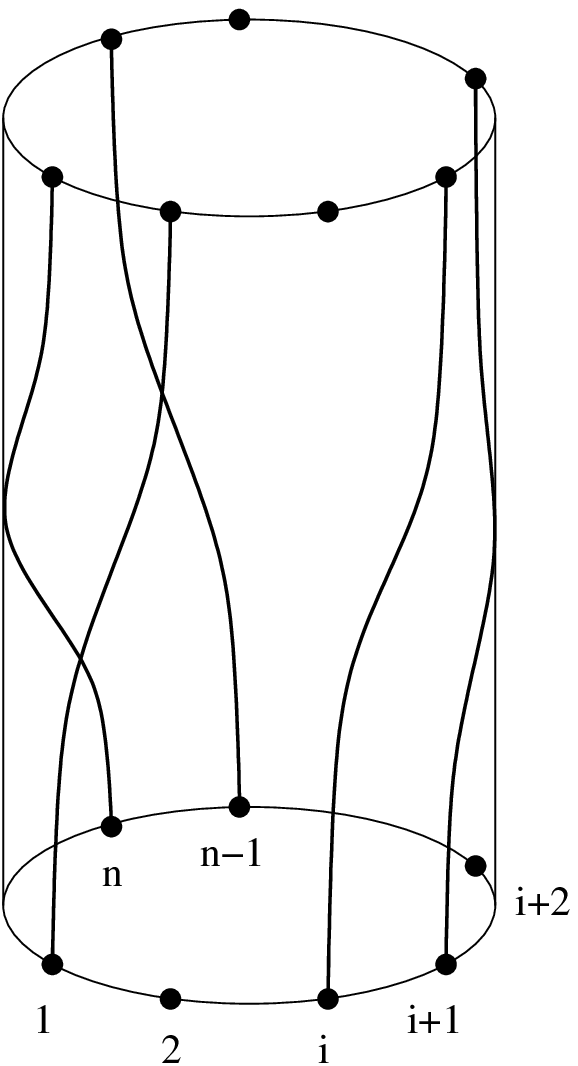}
   \caption{The affine braid generators $\sigma_i$ and $\rho$}
   \label{fig:braidcyl}
  \end{center}
\end{figure}

\begin{rems} The group~$\hat{\B}_{\hat{A}_{n-1}}$
\begin{itemize}
\item possesses as subgroups the finite type $A$ braid group~$\B_{A_{n-1}}$ generated by $\sigma_1, \dots, \sigma_{n-1}$, but also the affine type $A$ braid group~$\B_{\hat{A}_{n-1}}$ generated by $\sigma_1, \dots, \sigma_{n}$. In fact, $\hat{\B}_{\hat{A}_{n-1}}$ is simply isomorphic to the semi-direct product $\B_{\hat{A}_{n-1}} \rtimes \left< \rho \right>$ of the latter and of the infinite cyclic group generated by $\rho$,
where the action of $\rho$ on $\B_{\hat{A}_{n-1}}$ given by conjugation permutes cyclically the generators $\sigma_i$;
\item is isomorphic to the finite type $B$ braid group~$\B_{B_{n-1}}$ generated by $\sigma_1, \dots, \sigma_{n-1}$ and $\tau$ such that the~$\sigma_i$ are subject to the finite braid relations \eqref{BEA1} for $i = 1, \dots,n-1$ and \eqref{BEA2} for $i = 1, \dots,n-2$ and that the following relations are satisfied:
\begin{align} 
\sigma_i \tau & =  \tau \sigma_i & & \mbox{for} \ i=2,\dots, n-1 \label{BB1} \\
\tau \sigma_1 \tau \sigma_1 & = \tau \sigma_ {1} \tau \sigma_ {1}. & & \label{BB2} 
\end{align}
This isomorphism identifies the generators $\sigma_i$ for $i=1, \dots,n-1$ while it sends $\rho$ to the product $\tau \sigma_1 \dots \sigma_{n-1}$;
\item is a subgroup of the finite type $A$ braid group~$\B_{A_{n}}$ generated by the~$\sigma_i$ for~$i=0, \dots,n-1$ subject to the finite braid relations \eqref{BEA1} for $i = 0, \dots,n-1$ and \eqref{BEA2} for $i = 0, \dots,n-2$. It consists exactly in the subgroup generated by the elements of~$\B_{A_{n}}$ that leave the first strand (labelled by $0$) fixed. One hence recovers the cylindrical pictorial description of $\hat{\B}_{\hat{A}_{n-1}}$ by ''inflating'' this fixed strand that can be seen as the core of the cylinder, see e.g. Figure~\ref{fig:imrho} depicting the image $\sigma_0^2 \sigma_1 \dots \sigma_{n-1}$ of $\rho$ in $\B_{A_{n}}$. Note that the generator $\sigma_n$ is sent to $ \sigma_{0}^2 \sigma_{1} \dots \sigma_{n-2} \sigma_{n-1} \sigma_{n-2}^{-1} \dots \sigma_{1}^{-1} \sigma_{0}^{-2}$, and that in the type $B$ presentation of this group, $\tau$ is simply sent to $\sigma_0^2$.
\end{itemize}
\end{rems}

\begin{figure}[htbp]
  \begin{center}
  \psfrag{0}{$\scriptscriptstyle{0}$}
  \psfrag{1}{$\scriptscriptstyle{1}$}
  \psfrag{2}{$\scriptscriptstyle{2}$}
  \psfrag{i-1}{$\scriptscriptstyle{i-1}$}
  \psfrag{i}{$\scriptscriptstyle{i}$}
  \psfrag{i+1}{$\scriptscriptstyle{i+1}$}
  \psfrag{i+2}{$\scriptscriptstyle{i+2}$}
  \psfrag{n-2}{$\scriptscriptstyle{n-2}$}
  \psfrag{n-1}{$\scriptscriptstyle{n-1}$}
  \psfrag{n}{$\scriptscriptstyle{n}$}
   \includegraphics[height=4cm]{rhocyl.eps} 
   \hspace{1cm}
   \raisebox{2cm}{$\longmapsto$}
   \hspace{1cm}
   \includegraphics[height=4cm]{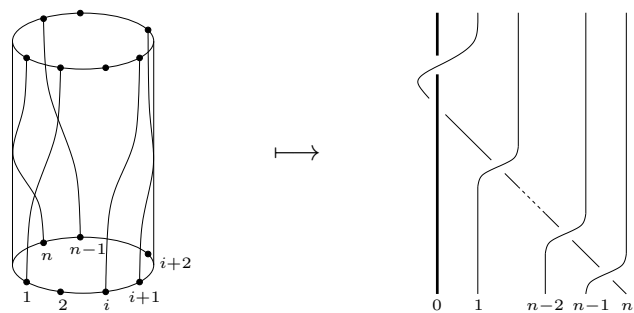}
   \caption{Image of $\rho$}
   \label{fig:imrho}
  \end{center}
\end{figure}

See \cite{All}, \cite{tD2}, \cite{KePei} or \cite{ChPei} for more details about this subject.

\subsection{Braid groups as mapping class groups}
\label{BraidMCG}

Let $M$ be an orientable surface possibly with boundary. We will denote by $\MCG(M, n+1)$ the mapping class group of the surface $M$ with $n+1$ marked points defined as the group of orientation-preserving homeomorphisms of $M$ that fix the $n+1$ marked points setwise and the boundary pointwise up to isotopy. We will use the notation $\Delta$ for the set $\{ 0, \dots, n\}$ of marked points and sometimes consider these marked points as punctures and view $M$ as a $n+1$-punctured surface.

For a fixed set $S \subset \Delta$, we will also consider the subgroup $\MCG(M, n+1, S)$ of $\MCG(M, n+1)$ consisting in all mapping classes fixing pointwise the punctures of $M$ belonging to the set $S$.

The finite type $A$ braid group~$\B_{A_{n}}$ is isomorphic to the mapping class group $\MCG(\D, n+1)$ of the $n+1$-punctured $2$-disk $\D$ depicted in Figure~\ref{fig:punctdisk}. 

\begin{figure}[htbp]
  \begin{center}
  \psfrag{D}{$\D$}
  \psfrag{dots}{$\dots$}
  \psfrag{0}{$\scriptscriptstyle{0}$}
  \psfrag{1}{$\scriptscriptstyle{1}$}
  \psfrag{2}{$\scriptscriptstyle{2}$}
  \psfrag{n-1}{$\scriptscriptstyle{n-1}$}
  \psfrag{n}{$\scriptscriptstyle{n}$}
   \includegraphics[height=3cm]{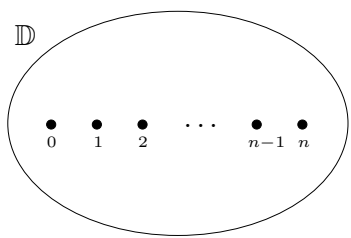} 
   \caption{The $n+1$-punctured $2$-disk $\D$}
   \label{fig:punctdisk}
  \end{center}
\end{figure}

Each generator $\sigma_i$ corresponds to the mapping class with support a small open disk enclosing the punctures $i$ and $i+1$ and consisting in rotating this disk by $\pi$ as described by Figure~\ref{fig:HTalpha}. We call this mapping class the half-twist along the arc $b_i$ and denote it by $t_{b_i}$. It swaps the punctures $i$ and $i+1$ and leaves all the others fixed pointwise.

\begin{figure}[htbp]
  \begin{center}
   \psfrag{i}{$\scriptscriptstyle{i}$}
   \psfrag{i+1}{$\scriptscriptstyle{i+1}$}
   \psfrag{alphai}{$\scriptstyle{b_i}$}
   \includegraphics[height=2cm]{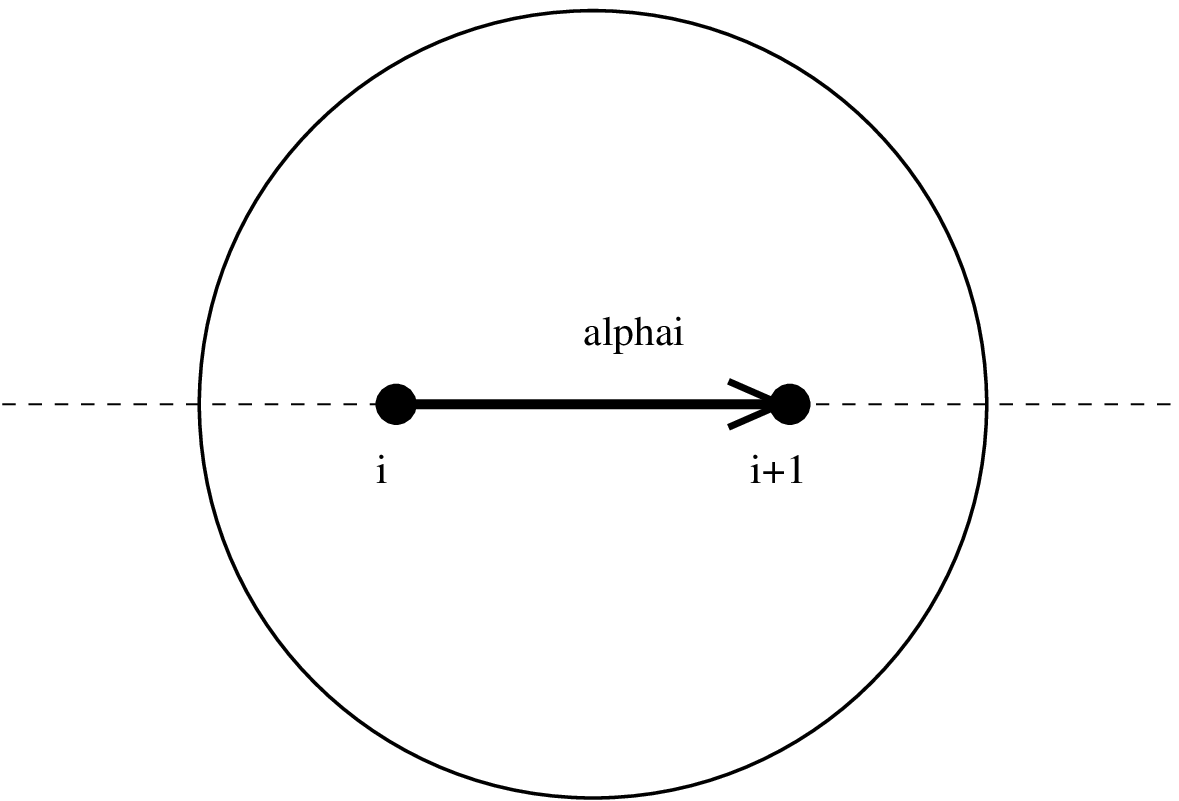} 
   \hspace{0.5cm}
   \raisebox{0.9cm}{$\longmapsto$}
   \hspace{0.5cm}
   \includegraphics[height=2cm]{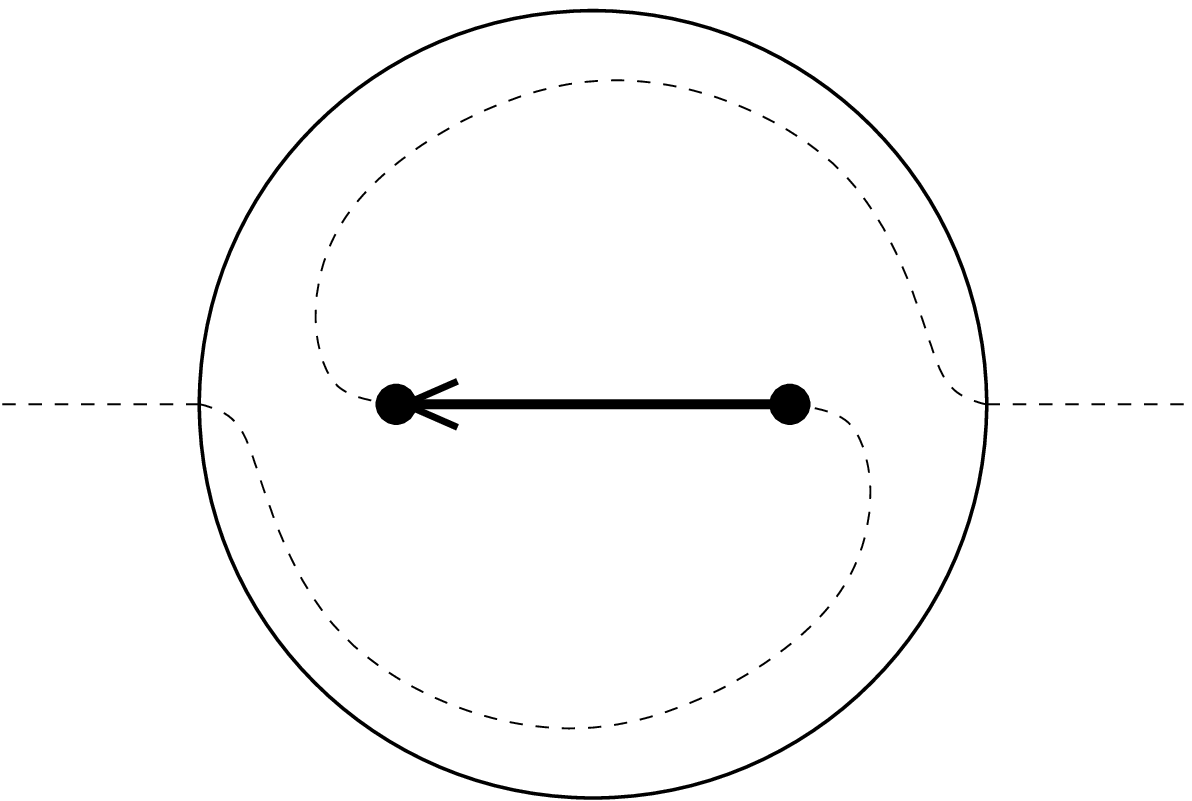}
   \caption{The half-twist along the arc $b_i$}
   \label{fig:HTalpha}
  \end{center}
\end{figure}

Its subgroup $\MCG(\D, n+1, \{ 0 \})$ is isomorphic to the finite type $B$ braid group~$\B_{B_{n-1}}$ where once again the half-twists $t_{b_i}$ are identified to the generators $\sigma_i$ for $i=1, \dots, n-1$, while $\tau$ corresponds to the full twist $t_{b_0}^2$. But $\B_{B_{n-1}}$ being isomorphic to $\hat{\B}_{\hat{A}_{n-1}} $, one might prefer to work with the extended affine $A$ presentation of this group. Then, to depict the generating mapping class corresponding to $\rho$, it is more convenient to draw the $n+1$-punctured $2$-disk as in Figure~\ref{fig:punctdisk2}. In this setting, $\rho$ will simply correspond to the $1/n$-twist $t_{\partial}$ with support an open disk enclosing all the punctures and consisting in rotating this disk by $\slfrac{2\pi}{n}$. It sends the $i$th puncture to the $i+1$st mod $n$, for $i=1, \dots, n$ while leaving the $0$th puncture fixed.

\begin{figure}[htbp]
  \begin{center}
  \psfrag{0}{$\scriptscriptstyle{0}$}
  \psfrag{1}{$\scriptscriptstyle{1}$}
  \psfrag{2}{$\scriptscriptstyle{2}$}
  \psfrag{n-1}{$\scriptscriptstyle{n-1}$}
  \psfrag{n}{$\scriptscriptstyle{n}$}
  \psfrag{i}{$\scriptscriptstyle{i}$}
  \psfrag{i+1}{$\scriptscriptstyle{i+1}$}
  \psfrag{i+2}{$\scriptscriptstyle{i+2}$}
    \psfrag{n-1}{$\scriptscriptstyle{n-1}$}
   \includegraphics[height=4cm]{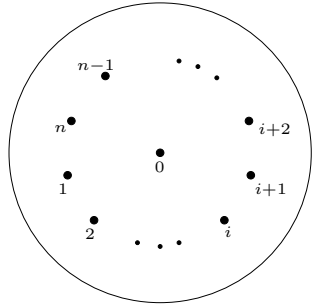} 
   \caption{The affine configuration of the $n+1$-punctured $2$-disk $\D$}
   \label{fig:punctdisk2}
  \end{center}
\end{figure}

\begin{rem} 
Let us consider the case where $M$ is the $n$-punctured annulus $\A$. Its mapping class group can be described as follows:
$$ \MCG(\A, n) \cong \B_{\hat{A}_{n-1}} \rtimes \Z^2 \cong \hat{\B}_{\hat{A}_{n-1}} \rtimes \Z$$
with $\Z^2 = \left< \mu, \eta \right> $ where $\mu$ corresponds to the $1/n$-twist $t_m$ with support an open annulus enclosing all the punctures  and consisting in rotating this annulus by $\slfrac{2\pi}{n}$ while $\eta$ corresponds to the full twist $t_c$ along a curve parallel to the central hole of $\A$. Note that when one collapses this central hole to a point, this last twist $t_c$ becomes trivial while $t_m$ equals $ t_{\partial}$ and one recovers $\hat{\B}_{\hat{A}_{n-1}}$ as $\MCG(\D, n+1, \{ 0 \})$. 
\end{rem}

In the sequel of this paper, we will sometimes use indistinctly the same notation for a braid and a homeomorphism representing its mapping class.

\section{Representations}

\subsection{Artin representation $\&$ braids as automorphisms of free groups}

When the finite type $A$ braid group $\B_{A_{n}}$ is topologically interpreted as the mapping class group $\MCG ( \D, n+1)$, it leads to a natural action on the fundamental group $\pi_1( \D,n+1, p)$ of the $n+1$-punctured $2$-disk. This latter group is the free group with $n+1$ generators $ F_{n+1}$ and is depicted on Figure~\ref{fig:punctdiskpi1}. 
\begin{figure}[htbp]
  \begin{center}
  \psfrag{D}{$\D$}
  \psfrag{b}{$\scriptstyle{p}$}
  \psfrag{dots}{$\dots$}
  \psfrag{0}{$\scriptscriptstyle{0}$}
  \psfrag{1}{$\scriptscriptstyle{1}$}
  \psfrag{2}{$\scriptscriptstyle{2}$}
  \psfrag{n-1}{$\scriptscriptstyle{n-1}$}
  \psfrag{n}{$\scriptscriptstyle{n}$}
  \psfrag{i}{$\scriptscriptstyle{i}$}
  \psfrag{i+1}{$\scriptscriptstyle{i+1}$}
  \psfrag{i+2}{$\scriptscriptstyle{i+2}$}
  \psfrag{x0}{$\scriptstyle{x_0}$}
  \psfrag{x1}{$\scriptstyle{x_1}$}
  \psfrag{x2}{$\scriptstyle{x_2}$}
  \psfrag{xn-1}{$\scriptstyle{x_{n-1}}$}
  \psfrag{xn}{$\scriptstyle{x_n}$}
   \includegraphics[height=4cm]{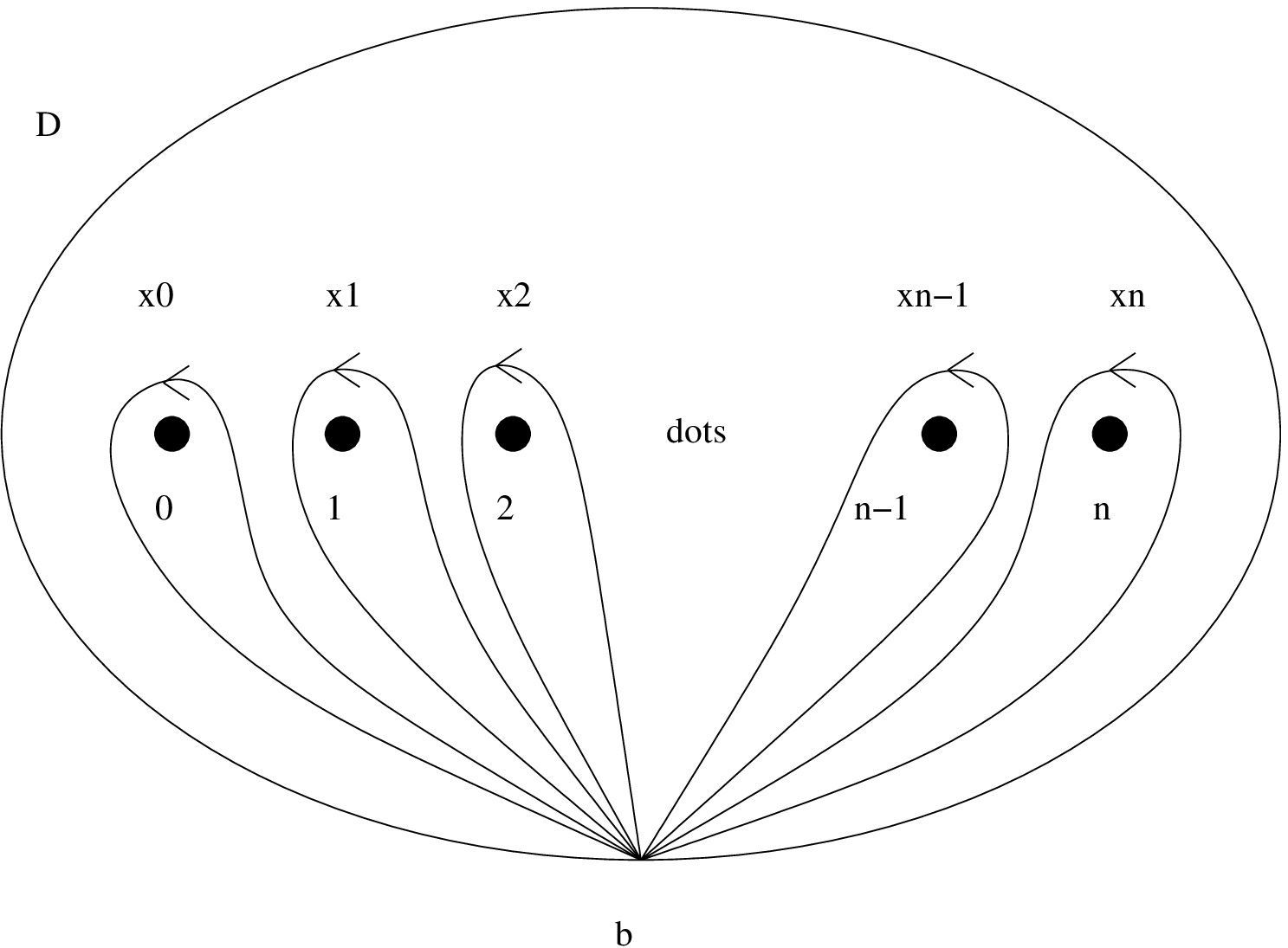} 
  \hspace{0.5cm}
      \raisebox{2cm}{$=$}
   \hspace{0.5cm}
    \includegraphics[height=4.5cm]{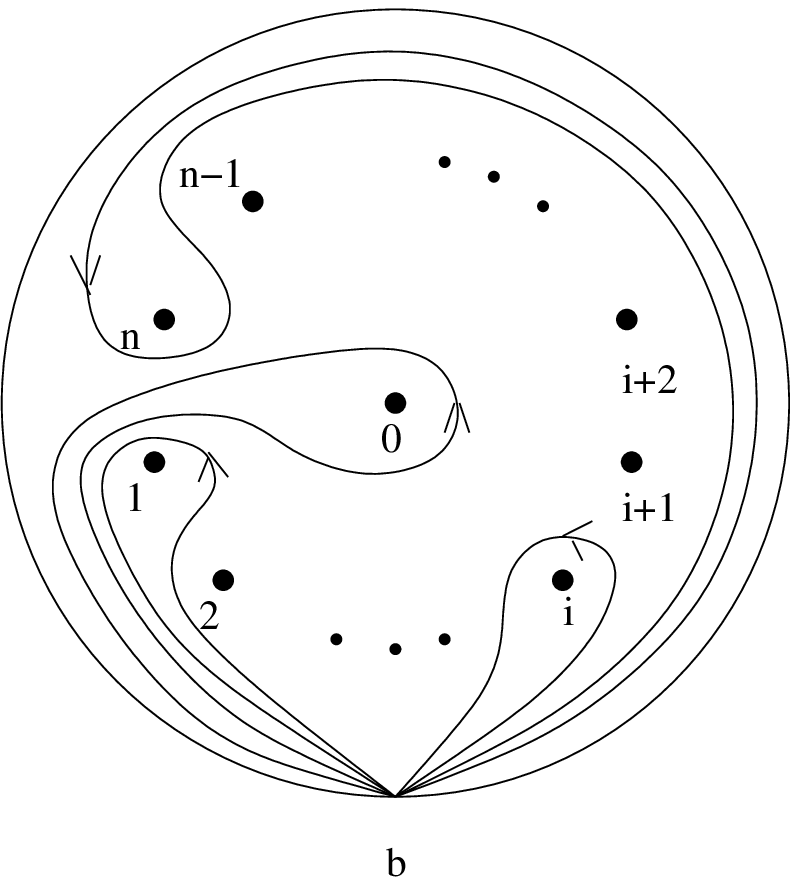} 
\caption{The fundamental group of the $n+1$-punctured $2$-disk $\D$}
   \label{fig:punctdiskpi1}
  \end{center}
\end{figure}
The action of the mapping class corresponding to $\sigma_i$ is described in Figure~\ref{fig:HTalphapi1}. 
\begin{figure}[htbp]
  \begin{center}
   \psfrag{b}{$\scriptstyle{p}$}
   \psfrag{xi}{$\scriptstyle{x_i}$}
   \psfrag{xi+1}{$\scriptstyle{x_{i+1}}$}
   \psfrag{i}{$\scriptscriptstyle{i}$}
   \psfrag{i+1}{$\scriptscriptstyle{i+1}$}
   \psfrag{alphai}{$\scriptstyle{\alpha_i}$}
   \includegraphics[height=2.5cm]{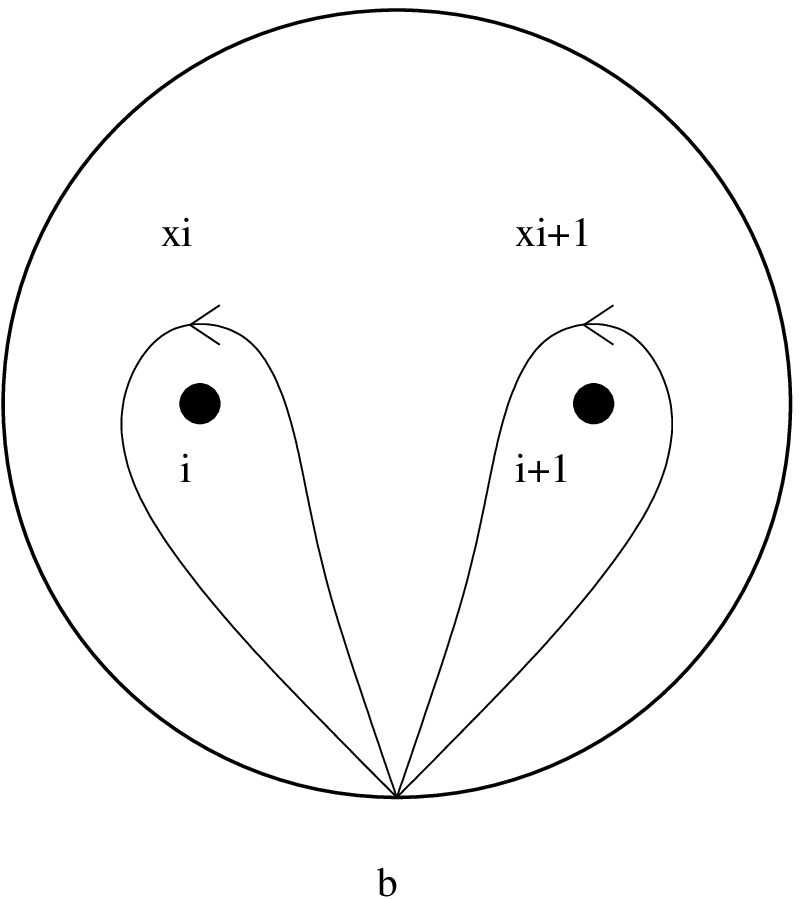} 
   \hspace{0.5cm}
   \raisebox{1.2cm}{$\longmapsto$}
   \hspace{0.5cm}
   \includegraphics[height=2.5cm]{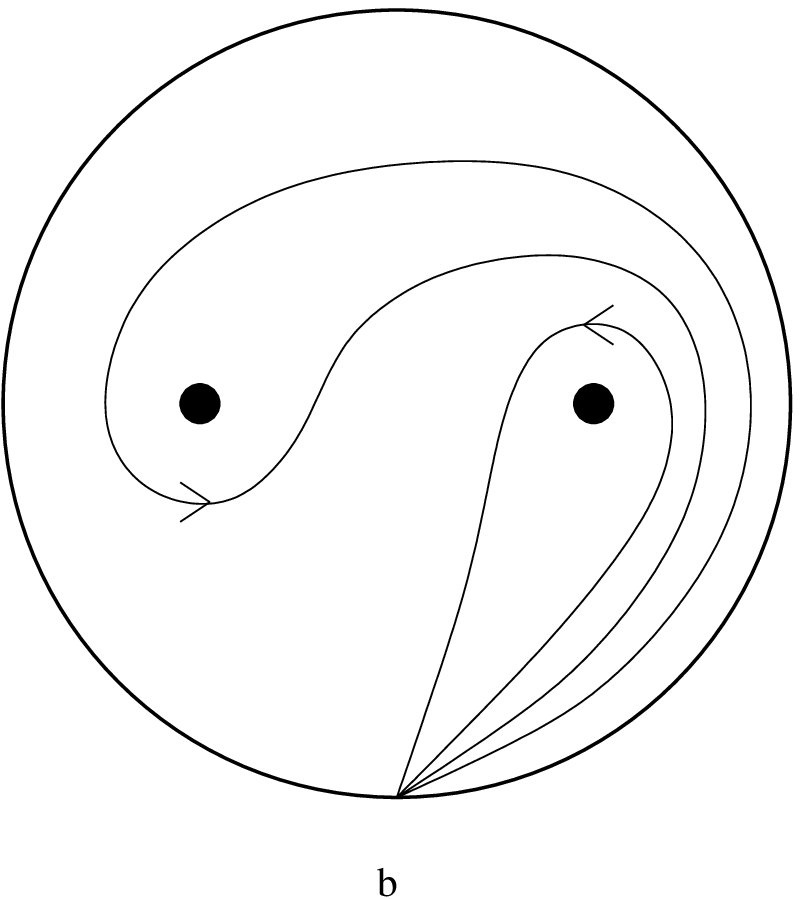}
   \caption{Topological description of the action of $\sigma_i$}
   \label{fig:HTalphapi1}
  \end{center}
\end{figure}
One hence obtains the well-known faithful Artin representation $\rho_A$ of $\B_{A_{n}}$ into the group of automorphisms of $ F_{n+1} = \left< x_0, \dots , x_{n}\right>$:
\begin{equation*}
\rho_A(\sigma_i)(x_j)  = 
\begin{cases}
x_{j+1} & \text{if $j=i$,}\\
x_j^{-1} x_{j-1} x_j & \text{if $j=i+1$,}\\
x_j & \text{otherwise}.
\end{cases}
\end{equation*}
In particular
\begin{equation*}
\rho_A(\sigma_i^{-1})(x_j)  = 
\begin{cases}
 x_j x_{j+1} x_j^{-1} & \text{if $j=i$,}\\
x_{j-1} & \text{if $j=i+1$,}\\
x_j & \text{otherwise}.
\end{cases}
\end{equation*}
Note that one reads a word in $\pi_1( \D,n+1, b)$ from right to left. In particular, the loop in $\pi_1( \D,n+1, b)$ that goes along the boundary counterclockwise is $ x_0 \dots x_n$.

When it comes to restricting $\rho_A$ to the subgroup $\hat{\B}_{\hat{A}_{n-1}}$ of $\B_{A_{n}}$, it is more convenient to picture the fundamental group $\pi_1( \D,n+1, p)$ in the affine configuration as shown also in Figure~\ref{fig:punctdiskpi1}. Let us just make explicit the images of the generators $\rho^{\pm 1}$ and $\sigma_n$ of $\hat{\B}_{\hat{A}_{n-1}}$ under the Artin representation:
\begin{align*}
\rho_A(\rho)(x_j) & =
\begin{cases}
x_1^{-1} x_{0} x_1 & \text{if $j=0$,}\\
x_n^{-1} \dots x_0^{-1} x_1 x_{0} \dots x_n & \text{if $j=n$,}\\
x_{j+1} & \text{otherwise}, \\
\end{cases} \\
\rho_A(\rho^{-1})(x_j) & =
\begin{cases}
x_{0} \dots x_n x_{n-1}^{-1} \dots x_1^{-1} x_{0} \dots x_{n-1} x_n^{-1} \dots x_0^{-1}& \text{if $j=0$,}\\
x_{0} \dots x_n x_{n-1}^{-1} \dots x_0^{-1}& \text{if $j=1$,}\\
x_{j-1} & \text{otherwise}, \\
\end{cases} \\
\rho_A(\sigma_n)(x_j) & = 
\begin{cases}
x_1^{-1} x_{0} \dots x_n x_{n-1}^{-1} \dots x_1^{-1} x_{0} \dots x_{n-1} x_n^{-1} \dots x_0^{-1} x_1 & \text{if $j=0$,}\\
x_1^{-1} x_{0} \dots x_n x_{n-1}^{-1} \dots x_0^{-1}  x_1  & \text{if $j=1$,}\\
x_n^{-1} \dots x_0^{-1} x_1 x_{0} \dots x_n & \text{if $j=n$,}\\
x_j & \text{otherwise}.
\end{cases}
\end{align*}

\subsection{Linear representation}

Out of these representations into the automorphisms group of the free group, one can construct, via Magnus expansion, linear representations of our braid groups. See \cite{JaBr}, \cite{BiMCG}, \cite{KaTu} for details about this procedure. Applied to $\rho_A$, the resulting linear representation is the famous Burau representation of $\B_{A_n}$ into $GL_{n+1} \left( \Z \left[ t^{\pm 1} \right] \right)$. 

But here we are interested in another way of defining this linear representation. The Burau representation (resp. its reduced version) can indeed be reconstructed homologically by considering the action of the finite type $A$ braid group $ \B_{A_n}$, viewed as mapping class group of the $n+1$--punctured disk, on the first relative homology group (resp. on the first homology group) of the infinite cyclic cover of $(\D,n+1)$ that has a structure of free $\Z \left[ t^{\pm 1} \right]$-module of rank $n+1$ (resp. of rank $n$). There exists a $n+1$ parameters generalisation of the Burau representation of $\B_{A_n}$, at the cost of restricting oneself to pure braids (the ones whose underlying permutation is the identity). The Gassner representation is then a linear representation of the pure braid group $\Pu_{A_n}$ into $GL_{n+1} \left( \Z \left[ t_1^{\pm 1}, \dots, t_n^{\pm 1} \right] \right)$. And similarly, the Gassner representation has a homological interpretation. 

For details on the finite case, see \cite{KaTu}, \cite{Oht}, 
 \cite{Abd}. We will focus here in defining an analogous homological linear representation but in the affine case. 

\subsubsection{Double infinite cyclic cover}\label{DG}
Consider the fundamental group $\pi_1( \D,n+1, p)$ of the $n+1$-punctured disk depicted in Figure~\ref{fig:punctdiskpi1} on which acts the extended affine type $A$ braid group $\hat{\B}_{\hat{A}_{n-1}} $ viewed as the mapping class group $\MCG(\D, n+1, \{ 0 \})$ via Artin representation. Let us consider the surjective group homomorphism $\phi_G$ from $\pi_1( \D,n+1, p)$ to $G = \Z \times \Z $ defined as follows 
$$ x_{i_1}^{m_1} x_{i_2}^{m_2} \dots x_{i_k}^{m_k} \mapsto \left( \sum_{ \substack{j = 1, \dots, k \\ i_j \neq 0}} m_j, \sum_{ \substack{j = 1, \dots, k \\ i_j = 0}} m_j \right) .$$
The cover $\tilde{\D}_G$, depicted in Figure~\ref{fig:cover}, associated to the kernel of this morphism has group of Deck transformations that is isomorphic to $G = \Z \times \Z   \cong \left< t \right> \times \left< q \right>$. Here we consider that the $n+1$-punctured disk, base of this covering, is displayed as in Figure~\ref{fig:punctdisk}. The homology group $H_1 \left( \tilde{\D}_G, G\tilde{p}, \Z \right)$ relative to the $G$--orbit of a fixed lift $\tilde{p}$ of the basepoint $p$ of $(\D,n+1)$ has a structure of $\Z \left[ t^{\pm 1}, q^{\pm 1} \right]$--module by inducing on the homology the action of $G$ on $\tilde{\D}_G$. As a $\Z \left[ t^{\pm 1}, q^{\pm 1} \right]$--module, it is free of rank $n+1$ and hence its automorphism group is isomorphic to $GL_{n+1} \left( \Z \left[ t^{\pm 1}, q^{\pm 1} \right] \right)$. While the homology group $H_1 \left( \tilde{\D}_G, \Z \right)$  is a free $\Z \left[ t^{\pm 1}, q^{\pm 1} \right]$--module of rank $n$ with automorphism group isomorphic to $GL_{n} \left( \Z \left[ t^{\pm 1}, q^{\pm 1} \right] \right)$. 

\begin{figure}[htbp]
  \begin{center}
  \psfrag{x0}{$\scriptstyle{\tilde{x}_0}$}
  \psfrag{x1}{$\scriptstyle{\tilde{x}_1}$}
  \psfrag{xn}{$\scriptstyle{\tilde{x}_n}$}
  \psfrag{b}{$\scriptscriptstyle{\tilde{p}}$}
  \psfrag{tb}{$\scriptscriptstyle{t\tilde{p}}$}
  \psfrag{qb}{$\scriptscriptstyle{q\tilde{p}}$}
  \psfrag{tqb}{$\scriptscriptstyle{tq\tilde{p}}$}
     \includegraphics[height=7cm]{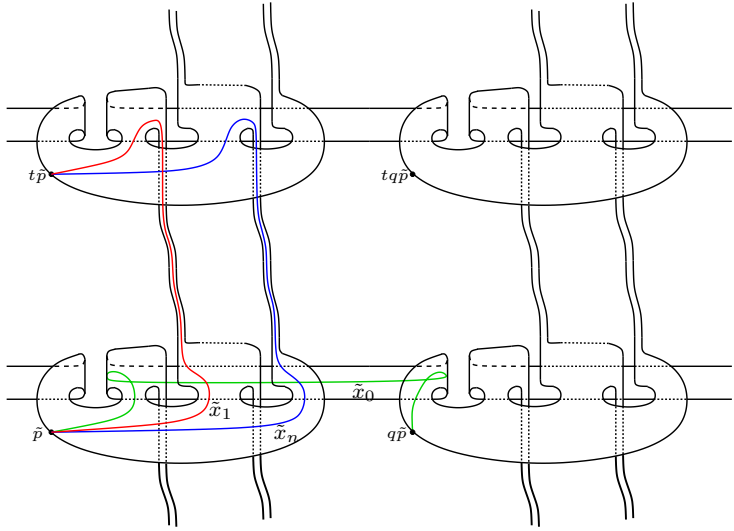} 
  \caption{The cover $\tilde{\D}_G$}
   \label{fig:cover}
  \end{center}
\end{figure}

Indeed recall that $n+1$-punctured disk $\D$ deformation retracts onto a bouquet of $n+1$ circles (or rose) generated by the loops $x_0, \dots, x_n$ based at $p$. Then the double infinite cyclic cover $\tilde{\D}_G$ deformation retracts onto the infinite graph $\Gamma_G$, see Figure~\ref{fig:gamma}. For $i= 1, \dots, n$, let $\tilde{x_i}$ be the lift of $x_i$ going from $\tilde{p}$ to $t\tilde{p}$ and let $\tilde{x_0}$ be the lift of $x_0$ going from $\tilde{p}$ to $q\tilde{p}$ depicted on Figure~\ref{fig:cover}. The set of vertices of $\Gamma_G$ is the $G$--orbit of $\tilde{p}$, i.e. $\{ t^kq^l\tilde{p} ; k,l \in \Z \}$ while its set of edges is given by $\{ t^kq^l\tilde{x_i} ; k,l \in \Z, i= 0, \dots, n \}$. So the simplicial complex associated to $\Gamma_G$ (resp. to the pair $\left(\Gamma_G, G\tilde{p}\right)$) is zero except in homological degree $1$ and $0$ (resp. degree $1$). Then it first follows directly that  
\begin{equation}
\label{baseH1rel}
 H_1 \left( \tilde{\D}_G, G\tilde{p}, \Z \right) = H_1 \left( \Gamma_G, G\tilde{p}, \Z \right) \cong \Z \left[ G \right] \left[ \left[ \tilde{x}_0 \right], \dots, \left[ \tilde{x}_n \right] \right] .
\end{equation}
where $\left[ \gamma \right]$ denotes the homology class of a path $\gamma$ in $\tilde{\D}_G$.
\begin{figure}[htbp]
  \begin{center}
    \psfrag{dots}{$\dots$}
     \includegraphics[height=4cm]{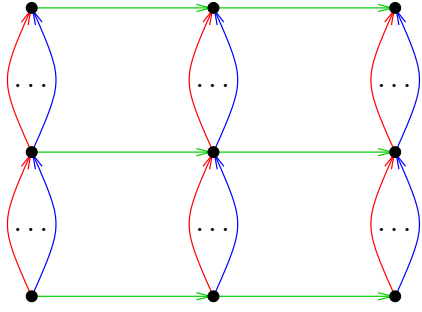} 
  \caption{The graph $\Gamma_G$}
   \label{fig:gamma}
  \end{center}
\end{figure}

On the other hand, the $\Z \left[ G \right]$--module $H_1 \left( \tilde{\D}_G, \Z \right)$ is free since the degree~$2$ chain of $\Gamma_G$ is zero. If one computes the kernel of the only non-zero differential of this simplicial complex, one obtains a basis for $H_1 \left( \tilde{\D}_G, \Z \right)$:
\begin{equation}
\label{baseH1}
H_1 \left( \tilde{\D}_G, \Z \right) = H_1 \left( \Gamma_G, \Z \right) \cong \Z \left[ G \right] \left[ \left[ \tilde{\gamma}_1 \right] , \dots, \left[ \tilde{\gamma}_{n-1} \right] ,\left[ \tilde{\gamma}_n \right] \right] 
\end{equation}
with
$$
 \left[ \tilde{\gamma}_i \right] =  \left[ \tilde{x}_{i+1} \right] - \left[ \tilde{x}_i \right] \quad \mbox{for} \ i=1,\dots, n-1,$$
the path $\tilde{\gamma}_i$ being $\tilde{x}_i^{-1}\tilde{x}_{i+1}$ in $\tilde{\D}_G$ and
$$
\left[ \tilde{\gamma}_n \right] =  (1-t) \left[ \tilde{x}_n \right] + t(1-t) \left[ \tilde{x}_{n-1} \right] + \dots + t^{n-1}(1-t) \left[ \tilde{x}_{1} \right] + t^{n}(1-t) \left[ \tilde{x}_{0} \right] +qt^n \left[ \tilde{x}_{1} \right] - \left[ \tilde{x}_n \right],$$
the path $\tilde{\gamma}_n$ being 
$$\tilde{x}_n^{-1}\left( t\tilde{x}_n\right)^{-1} \dots \left( t^{n}\tilde{x}_1\right)^{-1} \left( t^{n+1}\tilde{x}_0\right)^{-1} \left( qt^{n}\tilde{x}_{1} \right) \left(t^{n}\tilde{x}_0 \right) \left(t^{n-1}\tilde{x}_1 \right)\dots \tilde{x}_n$$
in $\tilde{\D}_G$.

Consider the group $H = \Z \cong  \left< t \right>$ and the group homomorphism $\phi_H$ which is the composite of $\phi_G$ by the projection from $G$ to $H$ that sends $(a,b)$ to $a+b$. The original Burau representation can be constructed by lifting the action of any mapping class to the corresponding cover $\tilde{\D}_H$ and then inducing to its first relative homology group. Note that this cover $\tilde{\D}_H$ is a quotient of the cover $\tilde{\D}_G$ we are working with. Here we mimic the homological construction of Burau but with this new cover $\tilde{\D}_G$. Before detailing this construction, let us also mention what happens if one plays the same game with the cover $\tilde{\D}_L$ associated to the abelianization map $\phi_L$, i.e. $L $ being $\Z^{n+1} \cong \left< t_0 \right> \times \dots \times \left< t_n \right>$, the first homology group of the $n+1$-punctured disk $\D$. In that case one obtains a representation of the subgroup consisting of all mapping classes which action commutes with $\phi_L$, namely the Torelli subgroup. In our case, where the surface is the disk $\D$, this Torelli subgroup is the pure braid group $\Pu_{A_n}$, and this homologically constructed representation is precisely the $n+1$ parameters $t_0, \dots, t_n$ Gassner representation.

\subsubsection{Homological representations}\label{rephomologique}

Now let us turn to the construction of the homological linear representation of $\hat{\B}_{\hat{A}_{n-1}}$ associated to the cover $\tilde{\D}_G$. For any braid $\sigma $ in $\hat{\B}_{\hat{A}_{n-1}}$ and loop $ x $ in $\pi_1( \D,n+1, p)$, the elements $x$ and $\rho_A (\sigma)(x)$ have same image under the homomorphism $\phi_G$ (this is not true for any braid in $\B_{A_n}$). This allows to construct a unique lift $\tilde{\sigma}$ acting on $\tilde{\D}_G$ of the mapping class $\sigma$ that fixes $\tilde{p}$ and commutes with the action of $G$ on $\tilde{\D}_G$ (and so fixes $G\tilde{p}$ pointwise). Let us denote by $\rho_H(\sigma)$ the automorphism of $H_1 \left( \tilde{\D}_G, G\tilde{p}, \Z \right) \cong \Z \left[ t^{\pm 1}, q^{\pm 1} \right]^{n+1}$ induced by $\tilde{\sigma}$, this provides an homological representation 
$$\rho_{H}: \hat{\B}_{\hat{A}_{n-1}} \rightarrow GL_{n+1} \left( \Z \left[ t^{\pm 1}, q^{\pm 1} \right] \right)$$

In fact we are looking for an homological representation of rank $n$ and not $n+1$, but let us still make $\rho_H$ explicit to understand how things work. For $i \neq 0, n$, the lift $\tilde{\sigma}_i$ leaves $\tilde{x}_{k}$ fixed for $k\neq i, i+1$, turns $\tilde{x}_{i}$ into $\tilde{x}_{i+1}$ and stretches $\tilde{x}_{i+1}$ to $\left( t\tilde{x}_{i+1}\right)^{-1} \left( t\tilde{x}_{i}\right) \tilde{x}_{i+1}$. One sees then immediately that it induces the following morphism $\rho_H(\sigma_i)$ on the homology: 
$$\left[ \tilde{x}_{k} \right] \mapsto \left[ \tilde{x}_{k}\right] \ \mbox{for} \ k \neq i, i+1, \quad \left[ \tilde{x}_{i} \right] \mapsto \left[ \tilde{x}_{i+1} \right] \ \mbox{and} \ \left[ \tilde{x}_{i+1} \right] \mapsto t \left[ \tilde{x}_{i} \right] + (1-t)\left[ \tilde{x}_{i+1} \right] $$
or written matricially in the basis given in \eqref{baseH1rel}:
\begin{equation*}
\rho_H(\sigma_i)  = 
\begin{pmatrix}
I_{i} & 0 & 0 & 0\\
0 & 0 & t & 0\\
0 & 1 & 1-t & 0 \\
0 & 0 & 0 & I_{n-1-i} \\
\end{pmatrix}
\end{equation*}
Similarly $\tilde{\sigma}_0^2$ leaves $\tilde{x}_{k}$ fixed for $k\neq 0, 1$, while it stretches $\tilde{x}_{0}$ to $\left( q\tilde{x}_{1}\right)^{-1} \left( t\tilde{x}_{0}\right) \tilde{x}_{1}$ and $\tilde{x}_{1}$ to $\left( t\tilde{x}_{1}\right)^{-1} \left( t^2\tilde{x}_{0}\right)^{-1}\left( tq\tilde{x}_{1}\right)\left( t\tilde{x}_{0}\right) \tilde{x}_{1}$ which induces the following on the homology:
\begin{equation*}
\rho_{H}(\sigma_0^2)  = 
\begin{pmatrix}
t & t(1-t) & 0 \\
1-q & 1+tq-t & 0 \\
0 & 0 & I_{n-1}  \\
\end{pmatrix} 
\end{equation*}
And finally $\tilde{\rho}$ sends $\tilde{x}_{k}$ to $\tilde{x}_{k+1}$ if $k\neq 0, n$, while it stretches $\tilde{x}_{0}$ to $\left( q\tilde{x}_{1}\right)^{-1} \left( t\tilde{x}_{0}\right) \tilde{x}_{1}$ and $\tilde{x}_{1}$ to 
$$\left( t\tilde{x}_{n}\right)^{-1} \left( t^2\tilde{x}_{n-1}\right)^{-1} \dots \left( t^n\tilde{x}_{1}\right)^{-1} \left( t^{n+1}\tilde{x}_{0}\right)^{-1}\left( qt^n\tilde{x}_{1}\right) \left( t^n\tilde{x}_{0}\right) \left( t^{n-1}\tilde{x}_{1}\right) \dots \left( t\tilde{x}_{n-1}\right) \tilde{x}_{n}$$
which induces the following on the homology:
\begin{equation*}
\rho_{H}(\rho)  = 
\begin{pmatrix}
t &  0 & 0 &\dots & 0 & t^n(1-t)\\
1-q & 0 & 0 &\dots & 0 & t^{n-1}(1-t+qt) \\
0 & 1 & 0 & \dots & 0 & t^{n-2}(1-t) \\
\vdots & \ddots & \ddots & \ddots & \vdots & \vdots \\
0 & \dots &  0 & 1  & 0 & t(1-t) \\
0 & \dots & \dots & 0 & 1 & 1-t \\
\end{pmatrix} 
\end{equation*}

Now let us make explicit the homological representation of dimension $n$ we are interested in. For any braid $\sigma $ in $\hat{\B}_{\hat{A}_{n-1}}$, its lift $\tilde{\sigma}$ acting on $\tilde{\D}_G$ induces an automorphism denoted $\rho_{RH}(\sigma)$  of $H_1 \left( \tilde{\D}_G, \Z \right) \cong \Z \left[ t^{\pm 1}, q^{\pm 1} \right]^{n}$ and hence the following homological representation:
$$\rho_{RH}: \hat{\B}_{\hat{A}_{n-1}} \rightarrow GL_{n} \left( \Z \left[ t^{\pm 1}, q^{\pm 1} \right] \right)$$
Let us again detail how the lifts of braid generators act on paths in the cover $\tilde{\D}_G$ and then what is the induced linear action on its homology. For $i=2, \dots, n-1$, the lift $\tilde{\sigma}_i$ leaves $\tilde{\gamma}_{k}$ fixed for $k\neq i-1,i, i+1$, turns the path $\tilde{\gamma}_{i} = \tilde{x}_{i}^{-1}\tilde{x}_{i+1}$ into $\left(\tilde{x}_{i+1}\right)^{-1}\left( t\tilde{x}_{i+1}\right)^{-1} \left( t\tilde{x}_{i}\right) \tilde{x}_{i+1}$, turns the path $\tilde{\gamma}_{i-1} =\tilde{x}_{i-1}^{-1}\tilde{x}_{i}$ into $ \tilde{x}_{i-1}^{-1}\tilde{x}_{i+1}$ and turns the path $\tilde{\gamma}_{i+1} =\tilde{x}_{i+1}^{-1}\tilde{x}_{i+2}$ into $\left(\tilde{x}_{i+1}\right)^{-1}\left( t\tilde{x}_{i}\right)^{-1} \left( t\tilde{x}_{i+1}\right) \tilde{x}_{i+2}$. One sees then immediately that it induces the following morphism $\rho_{RH}(\sigma_i)$ on the homology: 
$$\left[ \tilde{\gamma}_{k} \right] \mapsto \left[ \tilde{\gamma}_{k} \right] \ \mbox{for} \ k \neq i-1, i, i+1, \quad \left[ \tilde{\gamma}_{i-1} \right] \mapsto \left[ \tilde{\gamma}_{i-1} \right] + \left[ \tilde{\gamma}_{i} \right],$$
$$ \left[ \tilde{\gamma}_{i} \right] \mapsto -t \left[ \tilde{\gamma}_{i} \right] \ \mbox{and} \ \left[ \tilde{\gamma}_{i+1} \right] \mapsto \left[ \tilde{\gamma}_{i+1} \right] + t \left[ \tilde{\gamma}_{i} \right] $$
or matricially in the basis \eqref{baseH1}:
\begin{equation*}
 \rho_{RH}(\sigma_i)  = 
\begin{pmatrix}
I_{i-2} & 0 & 0 & 0 & 0\\
0 & 1 & 0 & 0 & 0 \\
0 & 1 &-t & t & 0\\
0 & 0 & 0 & 1 & 0 \\
0 & 0 & 0 & 0 & I_{n-1-i} \\
\end{pmatrix} \qquad \text{for } i = 2, \dots n-1,
\end{equation*}
The lift $\tilde{\sigma}_1$ acts similarly on $\tilde{\gamma}_{1}$ and $\tilde{\gamma}_{2}$, so we only have to observe that it turns the path 
$$\tilde{\gamma}_{n} =\tilde{x}_n^{-1}\left( t\tilde{x}_n\right)^{-1} \dots \left( t^{n}\tilde{x}_1\right)^{-1} \left( t^{n+1}\tilde{x}_0\right)^{-1} \left( qt^{n}\tilde{x}_{1} \right) \left(t^{n}\tilde{x}_0 \right) \left(t^{n-1}\tilde{x}_1 \right)\dots \tilde{x}_n$$
 into 
$$\tilde{x}_n^{-1}\left( t\tilde{x}_n\right)^{-1} \dots \left( t^{n}\tilde{x}_1\right)^{-1} \left( t^{n+1}\tilde{x}_0\right)^{-1} \left( qt^{n}\tilde{x}_{2} \right) \left(t^{n}\tilde{x}_0 \right) \left(t^{n-1}\tilde{x}_1 \right)\dots \tilde{x}_n$$
and hence induces the following morphism $\rho_{RH}(\sigma_1)$ on the homology: 
$$\left[ \tilde{\gamma}_{k} \right] \mapsto \left[ \tilde{\gamma}_{k} \right] \ \mbox{for} \ k \neq n, 1, 2, \quad \left[ \tilde{\gamma}_{n} \right] \mapsto \left[ \tilde{\gamma}_{n} \right] + qt^n \left[ \tilde{\gamma}_{1} \right],$$
$$ \left[ \tilde{\gamma}_{1} \right] \mapsto -t \left[ \tilde{\gamma}_{1} \right] \ \mbox{and} \ \left[ \tilde{\gamma}_{2} \right] \mapsto \left[ \tilde{\gamma}_{2} \right] + t \left[ \tilde{\gamma}_{1} \right] $$ 
i.e.
\begin{equation*}
 \rho_{RH}(\sigma_1)  = 
\begin{pmatrix}
-t & t & 0 & t^n q\\
0 & 1 & 0 & 0\\
0 & 0 & I_{n-3} & 0 \\
0 & 0 & 0 & 1 \\
\end{pmatrix}
\end{equation*}
Finally, it is obvious that the lift $\tilde{\rho}$ sends $\tilde{\gamma}_{k}$ to $\tilde{\gamma}_{k+1}$ for all $k \neq n$, while it turns $\tilde{\gamma}_{n}$ into 
$$\left( \tilde{x}_{n}\right)^{-1} \left( t\tilde{x}_{n-1}\right)^{-1} \dots \left( t^{n-1}\tilde{x}_{1}\right)^{-1} \left( t^{n}\tilde{x}_{0}\right)^{-1} \left( qt^{n}\tilde{x}_{1}\right)^{-1} \left( qt^n\tilde{x}_{2}\right) \left( t^n\tilde{x}_{0}\right) \left( t^{n-1}\tilde{x}_{1}\right) \dots \left( t\tilde{x}_{n-1}\right) \tilde{x}_{n}$$
so induces the morphism $\rho_{RH}(\rho)$ on the homology: 
$$\left[ \tilde{\gamma}_{k} \right] \mapsto \left[ \tilde{\gamma}_{k+1} \right] \ \mbox{for} \ k \neq n, \quad \left[ \tilde{\gamma}_{n} \right] \mapsto qt^n \left[ \tilde{\gamma}_{1} \right]$$
i.e.
\begin{equation*}
\rho_{RH}(\rho)  =
\begin{pmatrix}
0 &  t^n q \\
I_{n-1} & 0 \\
\end{pmatrix} 
\end{equation*}
The images of $\sigma_0^2$, $\sigma_n$ and $\rho^{-1}$ under $\rho_{RH}$ can be obtained using the ones already expressed and the affine braid relations.

\section{Action on a module category}\label{action}

In \cite{KhS}, Khovanov and Seidel construct a categorical representation of the finite type $A$ braid group $ \B_{A_{n}}$ in the homotopy category of the category of graded projective  left modules over a certain quotient of the path algebra of a finite type $A$ quiver. This representation is faithful and it decategorifies on a linear one parameter representation equivalent to the Burau representation of $\B_{A_n}$. Our aim is, following Khovanov and Seidel's ideas, to use an affine type $A$ quiver in order to obtain a categorical representation of the extended affine type $A$ braid group $\hat{\B}_{\hat{A}_n}$. This latter representation is designed to decategorify on the linear $2$-parameters representation $\rho_{RH}$ constructed in Section \ref{rephomologique}, and hence requires to work with a module category endowed with a rich algebraic structure, namely a trigrading. 

This section is devoted to the definitions of those affine quiver algebra, trigraded module category and categorical representation.

\subsection{The quiver algebra $R_n$}

The notation used here for paths is taken from \cite{KhS}. Start with the cyclic quiver $\Gamma_n$ pictured in Figure \ref{fig:quiverAaff}
\begin{figure}[!h]
  \begin{center}
  \psfrag{1}{$\scriptscriptstyle{1}$}
  \psfrag{2}{$\scriptscriptstyle{2}$}
  \psfrag{n-1}{$\scriptscriptstyle{n-1}$}
  \psfrag{n}{$\scriptscriptstyle{n}$}
  \psfrag{i}{$\scriptscriptstyle{i}$}
  \psfrag{i+1}{$\scriptscriptstyle{i+1}$}
  \psfrag{i+2}{$\scriptscriptstyle{i+2}$}
  \includegraphics[height=4cm]{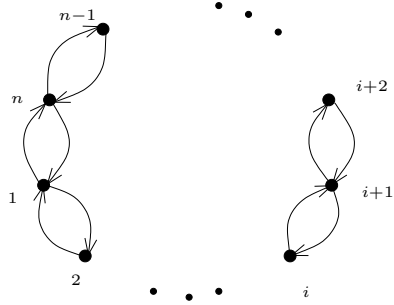} 
  \caption{The affine type $A$ quiver $\Gamma_n$}
   \label{fig:quiverAaff}
  \end{center}
\end{figure}
and let $R_n$ be the quotient of the path ring of the quiver $\Gamma_n$ by the relations:
\begin{align*}
 \left( i|i+1|i \right) & = \left( i|i-1|i \right) \qquad \text{for } i = 1, \dots n, \\
 \left( i-1|i|i+1 \right) & = \left( i+1|i|i-1 \right)= 0 \qquad \text{for } i = 1, \dots n, \\
\end{align*}
where the integers are taken modulo $n$.
This ring is trigraded and unital, with a family of mutually orthogonal idempotents $ \left( i \right)$ summing up to the unit element. The three gradings on $R_n$ are defined as follows:
\begin{itemize}
 \item the first grading is defined by setting that the degree of $ \left( i|i+1 \right)$ is one while the degree of any other generator is zero (which is the opposite convention as the one chosen by Khovanov and Seidel);
\item the second grading is simply the path length grading. Note that with the given relations, any path is at most of length $2$ in $R_n$, hence this second grading will be considered as a $\Z / 2\Z$ grading;
\item the third grading is defined by setting that the degree of $ \left( n|1 \right)$ is $1$, the degree of $ \left( 1|n \right)$ is $-1$ while the degree of any other element in $R_n$ is zero.
\end{itemize}
These three gradings are well-defined and one will denote by $\{ - \}$ a shift in the first grading, by $\left( - \right)$ a shift in the second and by $\left< - \right>$ a shift in the third. The convention being that the $i$th summand of a module shifted by $k$ is the $(i-k)$th summand of the original module.

As an abelian group $R_n$ is free of rank $4n$.

\begin{rem}
If one forgets about the two last gradings on $R_n$, and just consider it as a singly graded algebra, it is in fact a particular case of the general construction of algebras associated to graphs by Huerfano and Khovanov, see \cite{HK}. Note that, in this paper, they are also considering actions of quantum groups and braid groups on certain module categories over these algebras.
\end{rem}

The category of finitely generated trigraded left modules $R_n-\mbox{mod}$ has a Grothendieck ring $K\left( R_n-\mbox{mod} \right)$ which is isomorphic to $ \Z \left[ t^{\pm 1}, s^{\pm 1} \right] \ot \Z^{n}$.  The $ \Z \left[ t^{\pm 1} , s^{\pm 1} \right]$-module structure coming from the self--equivalences $ \{ 1 \}$ and $\left< 1 \right>$ of $R_n-\mbox{mod}$ consisting in shifting the first and third gradings by one.  The problem with this category is that the isomorphism classes of the indecomposable left projective modules $P_i = R_n(i)$ do not form a basis as in the finite Khovanov-Seidel case because $R_n$ has infinite global dimension. 

Note that in the sequel, we will denote the right indecomposable projective modules $_iP = (i) R_n$ and the isomorphism class of a module $M$ by $\left[ M \right]$.

Hence we will rather work over the category $R_n-\mbox{proj}$ of finitely generated trigraded projective left modules. This category, unlike $R_n-\mbox{mod}$, is not abelian, but only additive, though its split Grothendieck ring is also isomorphic to $ \Z \left[ t^{\pm 1}, s^{\pm 1} \right] \ot \Z^{n}$, with basis $\mathbf{f} =\left\{ \left[P_i \right], i=1,\dots,n \right\} $. While, as before, the first grading (resp. the third) decategorifies onto the $ \Z \left[ t^{\pm 1} \right]$-module (resp. $ \Z \left[ s^{\pm 1} \right]$-module) structure, the second grading, which is a $\Z / 2\Z$ grading, decategorifies as a sign (which will sometimes be denoted $\epsilon$). 

Let $t_{\rho}$ be the automorphism of the ring $R_n$ that sends any path $(i_1|i_2|\dots|i_k)$ to $(i_1+1|i_2+1|\dots|i_k+1)$. One can observe that this automorphism do not preserve the trigrading on $R_n$, but only the two first gradings. This implies that, if one constructs a bimodule $R_{n}^{\rho}$ by simply twisting the right action on the regular bimodule $R_n$ by $t_{\rho}$, i.e. $r \in R_n$ acts on $R_{n}^{\rho}$ on the right by multiplication by $t_{\rho}(r)$, the resulting bimodule is not trigraded anymore. So, in order to define a trigraded twisted bimodule, one cannot only twist the action on the regular bimodule $R_n$ but one has to construct a new bimodule in the more subtle way that we will describe now. 

Let $T_n^{\rho}$ be the trigraded $R_n$--bimodule generated by all elements of $R_n$ set to be in the same first and second degree as in $R_n$, but with the third grading shuffled as follows: 
\begin{itemize}
\item the degree of $ \left( 1 \right)$, $ \left( 2|1 \right)$, $ \left( 1|n \right)$ and $ \left( 1|2|1 \right)$ is $-1$
\item the degree of any other element is zero.
\end{itemize} 
The left action of $R_n$ on $T_n^{\rho}$ is simply the multiplication while its right action is the multiplication twisted by $t_{\rho}$. Let us also consider the trigraded $R_n$--bimodule $^{\rho}T_n$ constructed similarly but with the third grading shuffling given by: 
\begin{itemize}
\item the degree of $ \left( 1 \right)$, $ \left( n|1 \right)$, $ \left( 1|2 \right)$ and $ \left( 1|2|1 \right)$ is $1$
\item the degree of any other element is zero.
\end{itemize}
Here $R_n$ acts on the right on $^{\rho}T_n$ by multiplication and on the left by multiplication twisted by $t_{\rho}$.

\begin{lem}
 $T_n^{\rho}$ and $^{\rho}T_n$ are well-defined trigraded $R_n$--bimodules.
\end{lem}

\begin{proof}
 It is easy to check that, for any $r \in R_n$ and $a,b \in T_n^{\rho}$, one gets $\deg(r) + \deg (a) = \deg (r.a) = \deg (ra)$ and $\deg(r) + \deg (b) = \deg (b.r) = \deg (bt_{\rho}(r))$. And similarly for $^{\rho}T_n$.
\end{proof}

\begin{rem}
The chosen shufflings of the third grading appear to be natural when one observes that, as a trigraded left $R_n$--module, $T_{n}^{\rho}$ is simply isomorphic to $P_1 \left< -1 \right> \oplus P_{2}  \oplus \ldots \oplus P_n $ while, as a trigraded right $R_n$--module, $\vphantom{P}^{\rho}T_{n} $ is isomorphic to $\vphantom{P}_1P \left< 1 \right> \oplus \vphantom{P}_{2}P \oplus \ldots \oplus \vphantom{P}_nP$.
\end{rem}

The following lemma can be verified by simple computations that we omit here.

\begin{lem}\label{isostrig}
We have the following isomorphisms of trigraded $R_n$--bimodules 
$$T_{n}^{\rho} \ot_{R_n} \vphantom{P}^{\rho}T_{n} \cong  R_{n} \cong \vphantom{P}^{\rho}T_{n} \ot_{R_n} T_{n}^{\rho},$$
of trigraded left $R_n$--modules
$$ T_{n}^{\rho} \ot_{R_n} P_i \cong P_{i+1}$$
for all $i = 1, \dots, n-1$, and 
$$ T_{n}^{\rho} \ot_{R_n} P_n \cong P_{1}\left< -1 \right>$$
and of trigraded right $R_n$--modules
$$\vphantom{P}_iP  \ot_{R_n} \vphantom{P}^{\rho}T_{n} \cong \vphantom{P}_{i+1}P$$
for all $i = 1, \dots, n-1$, and 
$$\vphantom{P}_n P  \ot_{R_n} \vphantom{P}^{\rho}T_{n} \cong \vphantom{P}_{1}P\left< 1 \right>$$
In particular this implies that
$$ \vphantom{P}^{\rho}T_{n} \ot_{R_n} P_{i+1} \cong P_{i}$$
for all $i = 1, \dots, n-1$, and 
$$ \vphantom{P}^{\rho}T_{n} \ot_{R_n} P_1 \cong P_{n}\left< 1 \right>.$$
\end{lem}

\subsection{Categorical representation}

We consider $\cat_n$ to be the homotopy category of bounded cochain complexes of $R_n$--proj. Its Grothendieck ring is also isomorphic to $ \Z \left[ t^{\pm 1}, s^{\pm 1} \right] \ot \Z^{n}$, with basis the isomorphism classes of projectives, see \cite{Ro}. 

For all $i=1, \dots, n$, the two complexes $F_i $ and $F'_i$ of $R_n$--bimodules are defined as in \cite{KhS}:
\begin{align*}
F_i & : 0 \rightarrow  P_i \ot_{\Z} \vphantom{P}_iP \xrightarrow{d_i} R_n \rightarrow 0 \\
F'_i & : 0 \rightarrow  R_n \xrightarrow{d'_i}  P_i \ot_{\Z} \vphantom{P} _i P \{-1\}  \rightarrow 0 \\
\end{align*}
with $R_n$ sitting in cohomological degree zero and where the respective differentials of these length one complexes are:
\begin{align*}
d_i ((i)\ot(i)) & = (i) \\
d'_i (1) & = (i-1|i)\ot(i|i-1) +  (i+1|i)\ot(i|i+1) \\
 & + (i)\ot(i|i-1|i) + (i|i-1|i)\ot(i)\\
\end{align*}
where the integers again have to be understood modulo $n$.

Consider also the two following complexes of $R_n$--bimodules of length zero concentrated in cohomological degree zero:
\begin{align*}
F_{\rho} & : 0 \rightarrow  T_{n}^{\rho} \rightarrow 0 \\
F'_{\rho} & : 0 \rightarrow  \vphantom{P}^{\rho}T_{n}  \rightarrow 0 \\
\end{align*}

\begin{rem}\label{catetc}
Consider the functors 
$$\Fu_i = F_{i}  \ot_{R_n} -, \quad \Fu'_i = F'_{i}  \ot_{R_n} -, \quad \Fu_{\rho} = F_{\rho} \ot_{R_n} - \ \mbox{and} \ \Fu'_{\rho} = F'_{\rho} \ot_{R_n} -.$$
Since the bimodules $P_i \ot_{\Z} \vphantom{P}_iP$, $T_{n}^{\rho}$ and $\vphantom{P}^{\rho}T_{n}$ are projective as left modules, the former functors are well-defined endofunctors of the category $\cat_n$. Plus these bimodules being also projective as right modules, these functors are actually exact. Hence they induce linear maps on the Grothendieck ring $K\left(\cat_n\right)$. 
\end{rem}

\begin{prop}
 \hfill
\begin{itemize}
  \item[(i)] There is a weak action of the braid group $\hat{\B}_{\hat{A}_{n-1}}$ on $\cat_n$ given on the generators $\sigma_i$ by the functors $\Fu_{i}$, on their inverses $\sigma_i^{-1}$ by the functors $\Fu'_{i}$, on the generator $\rho$ by the functor $ \Fu_{\rho}$, on its inverse $\rho^{-1}$ by the functor $ \Fu'_{\rho} $ and on any braid word $\sigma$ by the functor $\Fu_{\sigma}$ consisting in tensoring on the left by the tensor product of the complexes associated to the generators appearing in the braid word.
  \item[(ii)] This action induces a linear representation $\rho_{AKS}$ of $\hat{\B}_{\hat{A}_{n-1}}$ on the Grothendieck ring $K\left(\cat_n\right) \cong  \Z \left[ t^{\pm 1} , s^{\pm 1}\right]^{n}$ which is given in the basis $ \mathbf{f} = \left\{ \left[ P_i \right], i = 1, \dots,n \right\}$ of $K\left(\cat_n\right)$ by:
\begin{align*}
\rho_{AKS}(\sigma_i) & = 
\begin{pmatrix}
I_{i-2} & 0 & 0 & 0 & 0\\
0 & 1 & 0 & 0 & 0\\
0 & 1 & -t & -t & 0 \\
0 & 0 & 0 & 1 & 0 \\
0 & 0 & 0 & 0 & I_{n-i-1} \\
\end{pmatrix} \qquad \text{for } i = 2, \dots n-1,\\
\rho_{AKS}(\sigma_{1}) & = 
\begin{pmatrix}
-t & t & 0 & s^{-1}  \\
0 & 1 & 0 & 0 \\
0 & 0 & I_{n-3} & 0 \\
0 & 0 & 0 & 1 \\
\end{pmatrix} \\
\rho_{AKS}(\sigma_{n}) & = 
\begin{pmatrix}
1 & 0 & 0 & 0 \\
0 & I_{n-3} & 0 & 0 \\
0 & 0 & 1 & 0 \\
ts & 0 & 1 & -t  \\
\end{pmatrix} \\
\rho_{AKS}(\rho) & = 
\begin{pmatrix}
 0 & s^{-1} \\
I_{n-1} & 0 \\
\end{pmatrix} \\
\rho_{AKS}(\rho^{-1}) & = 
\begin{pmatrix}
 0 & I_{n-1} \\
s & 0 \\
\end{pmatrix} \\
\end{align*}
\item[(iii)] The linear representation obtained by decategorification $\rho_{AKS}$ and the homological linear representation $\rho_{RH}$ are related as follows:
$$\rho_{AKS\vert s = q^{-1}t^{-n}} =  \rho_{RH}$$
\end{itemize}
\end{prop}

\begin{proof}
The first point $(i)$ follows from the results obtained by Khovanov and Seidel in the finite type $A$ case and from Lemma \ref{isostrig}.

The second point $(ii)$ is easy computations. Let us detail the case of $\sigma_1$ and $\rho$. The linear map $\rho_{AKS} (\sigma_1)$ is given by $  \left[ F_{1}  \ot_{R_n} - \right] =  \left[ Id \right]  - \left[P_{1} \ot_{\Z} \vphantom{P}_{1}P  \ot_{R_n} - \right] $. Hence the image of the basis $\mathbf{f} = \left\{ \left[ P_i \right], i = 1, \dots,n \right\}$ of $K\left(\cat_n\right)$ is
\begin{align*}
 \rho_{AKS} (\sigma_1)( \left[ P_1 \right]) & = \left[ P_1 \right] - \left[P_{1} \ot_{\Z} \vphantom{P}_{1}P  \ot_{R_n} P_1 \right] \\
& = \left[ P_1 \right] - \left[P_{1} \ot_{\Z} (1) \right] - \left[P_{1} \ot_{\Z} (1 | 2 | 1) \right] \\
& = \left[ P_1 \right] - \left[P_{1}\right] - t \epsilon^2 \left[P_{1} \right] \\
& = - t \left[P_{1} \right] \\
 \rho_{AKS} (\sigma_1)( \left[ P_2 \right]) & = \left[ P_2 \right] - \left[P_{1} \ot_{\Z} \vphantom{P}_{1}P  \ot_{R_n} P_2 \right] \\
& = \left[ P_2 \right] - \left[P_{1} \ot_{\Z} (1 | 2 ) \right] \\
& = \left[ P_2 \right]  - t \epsilon \left[P_{1} \right] \\
& = \left[ P_2 \right] + t \left[P_{1} \right] \\
 \rho_{AKS} (\sigma_1)( \left[ P_i \right]) & = \left[ P_i \right] - \left[P_{1} \ot_{\Z} \vphantom{P}_{1}P  \ot_{R_n} P_i \right] \\
& = \left[ P_i \right] \qquad \text{for } i = 3, \dots n-1 \\
 \rho_{AKS} (\sigma_1)( \left[ P_n \right]) & = \left[ P_n \right] - \left[P_{1} \ot_{\Z} \vphantom{P}_{1}P  \ot_{R_n} P_n \right] \\
& = \left[ P_n \right] - \left[P_{1} \ot_{\Z} (1 | n ) \right] \\
& = \left[ P_n \right]  - s^{-1} \epsilon \left[ P_{1} \right] \\
& = \left[ P_n \right] + s^{-1} \left[ P_{1} \right] 
\end{align*}
On the other hand, the linear map $\rho_{AKS} (\rho)$ is given by $  \left[ F_{\rho} \ot_{R_n} - \right] = \left[ T_{n}^{\rho}  \ot_{R_n} - \right] $, acting on our basis elements as follows (see Lemma \ref{isostrig}):
\begin{align*}
 \rho_{AKS} (\rho)( \left[ P_i \right]) & = \left[ T_{n}^{\rho}  \ot_{R_n} P_i \right] \\
& = \left[ P_{i+1} \right] \qquad \text{for } i = 1, \dots n-1 \\
 \rho_{AKS} (\rho)( \left[ P_n \right]) & = \left[ T_{n}^{\rho}  \ot_{R_n} P_n \right] \\
& = \left[ P_{1} \left< -1 \right> \right]\\
& = s^{-1} \left[ P_{1} \right].\\
\end{align*}

The third point $(iii)$ is trivial.
\end{proof}

\section{Geometric and trigraded intersection numbers}

In this section, the objects considered and their properties are modelled on the ones Khovanov and Seidel introduced in the Section $3$ of \cite{KhS}, but in a slightly more general setting. In order for this paper to be self-contained, we give here the (variants of) their definitions and results and precise the proofs when theirs do not apply straightforwardly to our affine setting.

\subsection{Geometric intersection numbers.}

The curves we will consider are oriented, though sometimes we will forget about their orientation. They are either simple and closed or the image of the embedding of a segment whose endpoints are sent to a pair of distinct marked points or points on the boundary of $\D$. Two curves are isotopic if one is the image of the other under an element of the mapping class group $\MCG(\D, n+1, \{ 0 \})$, isotopy of curves is denoted by $\simeq$. Two curves are said to have minimal intersection if their intersection points do not form an empty bigon (i.e. containing no marked point) unless these two intersection points are marked points. If $c_0$ and $c_1$ do not have minimal intersection, one can always find $c_1' \simeq c_1$ such that $c_0$ and $c_1'$ have minimal intersection. Hence we can define geometric intersection numbers as follows:
\begin{equation}\label{geomintnumb}
I(c_0,c_1)  = 
\begin{cases}
2 \ \text{if $c_0$, $c_1$ are closed and isotopic,}\\
\# \{c_0 \cap c_1' \backslash \Delta + \frac{1}{2} \# \{c_0 \cap c_1' \cap \Delta \} \\
\text{if $c_0$, $c_1$ do not intersect on $\partial \D$,}\\
I(c_0^+,c_1) \ \text{otherwise,}
\end{cases}
\end{equation}
where $c_0^+$ is obtained by slightly pushing $c_0$ along the flow of a vector field $Z$ that is obtained by extending a positively oriented vector field on $\partial \D$ to a smooth vector field on $\D$ that vanishes on $\Delta$.

\begin{rem}\label{syminvgeomintnb}
Note that in the two first cases, the geometric intersection number is symmetric, which is no longer true in the third case. Note also that the geometric intersection numbers are mapping class-invariant i.e. for any mapping class $\sigma$ and curves $c_0, c_1$, we have
$$I(\sigma(c_0),\sigma(c_1))= I(c_0,c_1).$$

\end{rem}

Consider the basic set of curves $b_i$ depicted on Figure \ref{fig:arcsbi}. A curve $c$ is called admissible if there exist a mapping class $\sigma$ and $i \in \{ 1 , \dots, n\}$ such that $c=\sigma(b_i)$. Conversely any curve whose endpoints lie in $\{ 1 , \dots, n\}$ is admissible.

\begin{figure}[htbp]
  \begin{center}
  \psfrag{0}{$\scriptscriptstyle{0}$}
  \psfrag{1}{$\scriptscriptstyle{1}$}
  \psfrag{2}{$\scriptscriptstyle{2}$}
  \psfrag{n}{$\scriptscriptstyle{n}$}
  \psfrag{i}{$\scriptscriptstyle{i}$}
  \psfrag{i+1}{$\scriptscriptstyle{i+1}$}
  \psfrag{ai}{$\scriptscriptstyle{b_i}$}
  \psfrag{an}{$\scriptscriptstyle{b_n}$}
  \psfrag{a1}{$\scriptscriptstyle{b_1}$}
   \includegraphics[height=4cm]{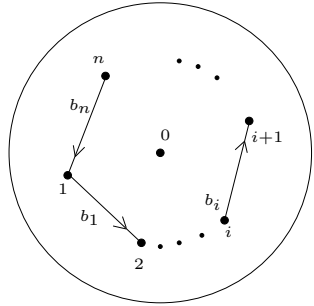} 
   \caption{The arcs $b_i$}
   \label{fig:arcsbi}
  \end{center}
\end{figure}

From now on and until mention of the contrary, we forget the orientation of the curves.

\begin{lem}\label{trivact->center}
If a mapping class $\sigma \in \MCG(\D, n+1, \{ 0 \})$ satisfies $\sigma (b_i) \simeq b_i$ for all $i = 1, \dots, n$, then $\sigma$ is in $\left< \rho^n \right>$, the center of $\MCG(\D, n+1, \{ 0 \})$.
\end{lem}

Recall that under the isomorphism $ \MCG(\D, n+1, \{ 0 \}) \simeq \hat{\B}_{\hat{A}_{n-1}} $, the mapping class $t_{\partial}$ corresponds to the braid $\rho$ so we might use the two notations indistincly.

\begin{proof} 
The restriction of such a mapping class to an annulus enclosing the punctures $\{1, \dots, n\}$ is equal to the identity. Moreover the mapping class group of a one-punctured disk being trivial, it actually implies that $\sigma$ restricted to a disk enclosing the punctures $\{0, \dots, n\}$ is the identity. Hence $\sigma$ is equal to $t_{\partial}^{pn}$ for some $p \in \Z$.
\end{proof}

\begin{lem}\label{geomintnb}
If $c$ is admissible and there exists $i\in \{1, \dots, n\}$ such that $I(b_j,c) = I(b_j,b_i)$ for all $j = 1, \dots, n$, then $c$ is isotopic either to $b_i$, to $a_i$ or to $a_i'$, see Figure \ref{fig:curvesai}.
\end{lem}

\begin{figure}[htbp]
  \begin{center}
  \psfrag{0}{$\scriptscriptstyle{0}$}
  \psfrag{1}{$\scriptscriptstyle{1}$}
  \psfrag{2}{$\scriptscriptstyle{2}$}
  \psfrag{n}{$\scriptscriptstyle{n}$}
  \psfrag{i}{$\scriptscriptstyle{i}$}
  \psfrag{i+1}{$\scriptscriptstyle{i+1}$}
  \psfrag{ai}{$\scriptscriptstyle{a_i}$}
  \psfrag{ai2}{$\scriptscriptstyle{a_i'}$}
  \psfrag{an}{$\scriptscriptstyle{b_n}$}
  \psfrag{a1}{$\scriptscriptstyle{b_1}$}
   \includegraphics[height=4cm]{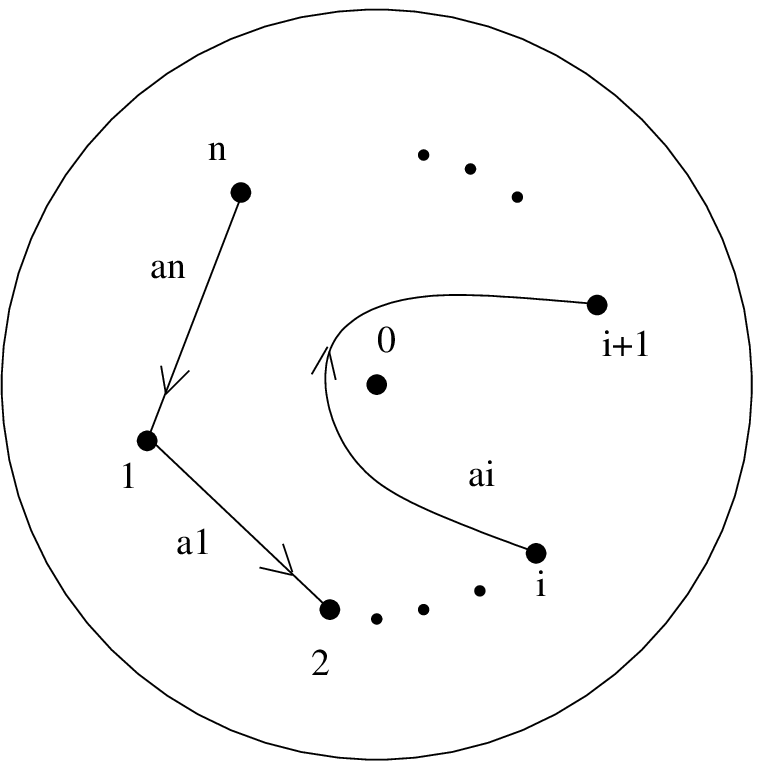} 
  \hspace{1cm}
  \includegraphics[height=4cm]{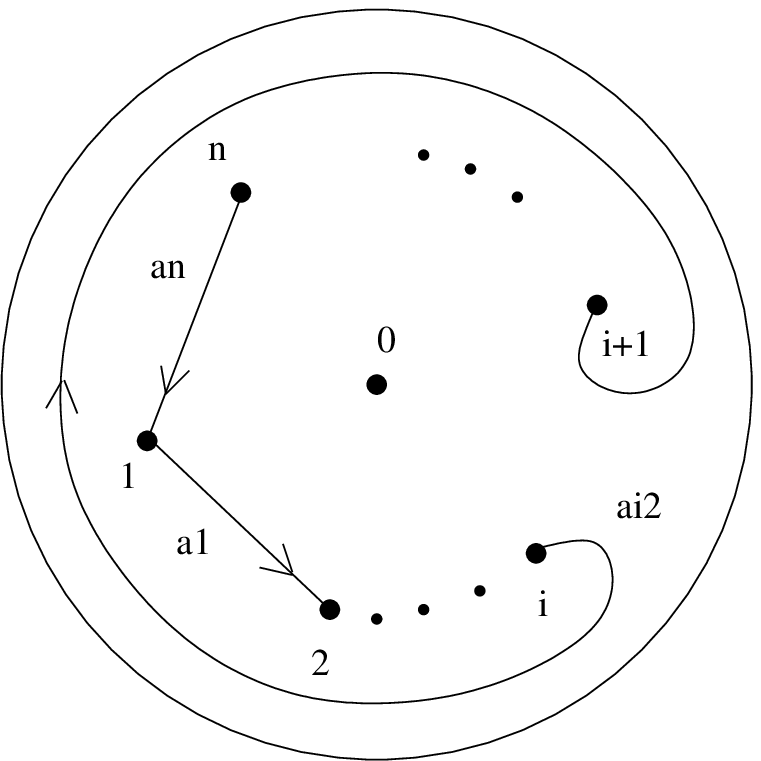} 
\caption{The curves $a_i$ and $a_i'$}
   \label{fig:curvesai}
  \end{center}
\end{figure}

\begin{proof}
Since $I(b_{i\pm1}, c) = I((b_{i\pm1}, b_i)= 1/2$, the curves $c$ and $b_{i\pm1}$ have exactly one endpoint in common and do not intersect else. Moreover, the assumption $I(b_{i\pm2}, c) = I((b_{i\pm2}, b_i)= 0$ implies that these latter endpoints must be the ones of $b_i$. Again these two points are the only intersection points in between $c$ and $b_i$. Finally the nullity of the remaining intersection numbers forces $c$ to also have no intersection with the $b_j$ for $ j \neq j, j\pm 1$. Hence $c$ is necessarly isotopic to $b_i$, $a_i$ or $a_i'$.
\end{proof}

\begin{rem} The curve $a_i$ is isotopic to $\rho^i \sigma_0^2 \rho^{-i}$ while the curve $a_i'$ is isotopic to $\rho^{i+1} \sigma_{n-1} \dots \sigma_2 \sigma_1 \rho^{-i}$.
\end{rem}


\begin{lem}\label{geomintnb2}
If a mapping class $\sigma \in \MCG(\D, n+1, \{ 0 \})$ satisfies $I(b_j,\sigma (b_i)) = I(b_j,b_i)$ for all $j,i = 1, \dots, n$, then $\sigma = \rho^{pn}$ for some $p\in \Z$.
\end{lem}

\begin{proof}
By Lemma \ref{geomintnb}, the curves $\sigma(b_i)$ are all isotopic to some explicit predicted curves. If there exist a $b_i$ such that $\sigma(b_i) \simeq a_i$ (resp. $\simeq a_i'$) then no other $b_j$ is such that $\sigma(b_j) \simeq a_j$ (resp. $\simeq a_j'$). Indeed if such a $b_j$ existed then we would have $I(\sigma(b_j),\sigma (b_i)) = 2$ if $j \neq i\pm1$ or $3/2$ if $j = i\pm1$. This contradicts the fact that $I(b_j,b_i) = 0$ if $j \neq i\pm1$ or $1/2$ if $j = i\pm1$ and the invariance of geometric intersection numbers, see Remark \ref{syminvgeomintnb}. Now let us rule some of the remaining possibilities out:

\noindent $\bullet$ Suppose $\sigma(b_i) \simeq a_i$ and $\sigma(b_j) \simeq b_j$ for all $j \neq i$. Consider the closed curve $c$ consisting in the concatenation $b_n \dots b_1$. The curve $c$ borders a disk with one marked point, hence so should do the image of $c$ under any mapping class. But here $\sigma(c)$ borders an empty disk, so this case cannot occur.

\noindent $\bullet$ The case where $\sigma(b_i) \simeq a_i'$ and $\sigma(b_j) \simeq b_j$ for all $j \neq i$ can be ruled out using the exact same argument.

\noindent $\bullet$ Now we treat the case where there are two integers $i$ and $j$ such that $\sigma(b_i) \simeq a_i$ and $\sigma(b_j) \simeq a_j'$. In the two previous cases, we used in disguise the fact the winding number around the central puncture is preserved by mapping classes. To talk about winding numbers, we now have to remember the orientation of curves. Let $c$ be as before with the orientation induced by the ones of the $b_i$'s as in Figure \ref{fig:arcsbi}. For continuity reasons, if $k\neq i,j$ the curve $\sigma(b_k) \simeq b_k$ has to carry the same orientation as $b_k$ while $\sigma(b_i) \simeq a_i$ and $\sigma(b_j) \simeq a_j'$ have to be oriented as in Figure \ref{fig:curvesai}. This leads then to a contradiction, the curve $c$ having a winding number around the central puncture equal to $1$ while the one of $\sigma(c)$ is $-1$.

\noindent $\bullet$ So we are left with only one possibility which is that $\sigma(b_i) \simeq b_i$ for all $i=1, \dots, n$ and we use Lemma \ref{trivact->center} to conclude.
\end{proof}

\subsection{Tangent bundle and trigraded intersection numbers}
\label{PTD}

Mimicking~\cite{KhS} we will now consider the real projectivization~$P = P T(\D \bs \Delta)$ and a covering of it with deck transformation group~$\Z^3$. 
Consider an oriented embedding of~$\D$ as an open subset of $\R^2$ so thatits tangent bundle~$T \D$ has a canonical oriented trivialization. 
As a consequence the projectivization of~$T \D$ in restriction over $\D \bs \Delta$ identifies to
$$P T(\D \bs \Delta) = \R P^1 \times (\D \bs \Delta).$$
For any puncture $i$ in $\Delta$, we will denote by $\lambda_i : S^1 \rightarrow \D \bs \Delta$ the 
choice of a small loop winding positively around the puncture $i$. 
As the classes $[point \times \lambda_i]$ together with the class of a fibre $[\R P^1 \times point]$ form a 
basis of $H_1(P; \Z)$, we define a class $C \in H^1(P; \Z^3)$ by specifying its images on these elements, namely:
\begin{align*}
C([point \times \lambda_i]) & =  (-2, 1, 0)  \qquad \mbox{for} \ i=1,\dots, n \\
C([point \times \lambda_0]) & =  (-2, 0, 1) \\
C([\R P^1 \times point]) & =  (1, 0, 0).&
\end{align*}
We will denote by~$\check{P}$ the covering classified by~$C$ and by~$\chi$ the~$\Z^3$-action on it.

\begin{rem} In this section, we will work with the group~$\Diff(\D, \Delta, \{0\})$ of smooth orientation preserving diffeomorphisms of~$\D$ which fix the boundary of the disk pointwise, preserve the set~$\Delta$ and fix the point~$\{0\}$, instead of the corresponding group of homeomorphisms as in Section~\ref{BraidMCG}. 
This is possible as any orientation-preserving homeomorphism that fixes the boundary of $\D$ and the point $\{ 0 \}$ pointwise and $\Delta \bs \{0\}$ setwise is isotopic to a diffeomorphism in~$\Diff(\D, \Delta, \{0\})$. Moreover~$\MCG(\D, n+1, \{ 0 \})$ is also equal to the group~$\Diff(\D, \Delta, \{0\})$ up to smooth isotopy (see 
for example~\cite[Section 1.4]{FaMaMCG}). 
As we will prove results that will not depend of the diffeomorphisms inside their isotopy class, we will abuse notation and use the same letters for a diffeomorphism and its class in~$\MCG(\D, n+1, \{ 0 \})$. In particular, the twists~$t_{b_i}$ and~$t_{\partial}$ may denote in the following, depending on the context, the mapping class group represention of the generators of the braid group as in Section~\ref{BraidMCG} or smooth diffeomorphisms in~$\Diff(\D, \Delta, \{0\})$ representing them.
\end{rem}

Let~$f$ be an element of~$\Diff(\D, \Delta, \{0\})$, its differential~$Df$ is a diffeomorphism 
of the tangent bundle of $\D \bs \Delta$ which is linear in the fibres of $T(\D \bs \Delta)$ and thus induces a diffeomorphism~$PDf$ of~$P$. 
As such a map~$f$ preserves winding numbers and as $Df$ sends a fibre to another fibre, the map~$PDf$ preserves the class $C$ and can be lifted to an equivariant diffeomorphism of~$\check{P}$.
We will denote by~$\check{f}$ the unique lift of~$PDf$ which acts trivially on the fibre of~$\check{P}$ over the points 
of~$P_{|\partial \D}$.

Note that any curve~$c$ has a canonical section~$s_c : c\bs \Delta \rightarrow P$ by taking the class in each fiber of
its tangent line:~$s_c(z) = [T_z c]$.
One defines a trigrading of~$c$ to be a lift~$\check{c}$ of~$s_c$ to~$\check{P}$ and a trigraded curve to be a pair $(c,\check{c}$) 
of a curve and a trigrading of that curve.
The $\Z^3$-action on~$\check{P}$ induces a $\Z^3$-action on the set of trigraded curves and the lifts of diffeomorphisms induce a $\Diff(\D, \Delta, \{0\})$-action on this set that commutes to the $\Z^3$-action. One can also lift the isotopy relation so that these actions 
induce actions of these same group on the set of isotopy classes of trigraded curves.

The arguments of~\cite{KhS} adapt to our trigraded case so that we have the following properties:

\begin{lem}\label{lem3.12-3.14}
\hfill
\begin{itemize}
  \item[(i)]  A curve~$c$ admits a trigrading if and only if it is not a simple closed curve.
  \item[(ii)] The~$\Z^3$-action on the set of isotopy classes of trigraded curves is free : a trigraded curve~$\check{c}$ is 
never isotopic to~$\chi(r_1,r_2,r_3) \check{c}$ for any~$(r_1,r_2,r_3) \neq 0$.
  \item[(iii)] Let~$c$ be a curve which joins two points of~$\Delta$, none of them being the puncture~$0$, let 
$t_c \in \Diff(\D, \Delta, \{0\})$ be the half twist along it and~$\check{t_c}$ its preferred lift to~$\check{P}$.
Then~$\check{t_c}(\check{c}) = \chi(-1,1,0) \check{c}$ for any trigrading~$\check{c}$ of~$c$.
\end{itemize}
\end{lem}

We now define some local index of intersection of curves just as in~\cite{KhS} but in our now trigraded case.
Let~$(c_0, \check{c}_0)$ and~$(c_1, \check{c}_1)$ be two trigraded curves, and let~$z \in \D \bs \partial \D$ be a point where $c_0$ and $c_1$ intersect transversaly. To define a local intersection index at~$z$ for the two curves one considers a small circle~$\ell \subset \D \bs \Delta$ around~$z$, and take~$\alpha : [0,1] \rightarrow \ell$ to be an embedded arc which moves clockwise along~$\ell$ 
such that~$\alpha(0) \in c_0$ and~$\alpha(1) \in c_1$ and for~$t \in ]0,1[$, 
$\alpha(t) \not \subset c_0 \cup c_1$. If~$z$ is a puncture, then~$\alpha$ is unique up to a change of 
parametrization, otherwise, there are two possibilities which are distinguished by their endpoints. Then take a 
smooth path~$\pi : [0,1] \rightarrow P$ with~$\pi(t) \in P_{\alpha(t)}$ for all $t$, 
from~$\pi(0) = T_{\alpha(0)} c_0$ to~$\pi(1) = T_{\alpha(1)} c_1$, and such 
that~$\pi(t) \neq T_{\alpha(t)} \ell$ for all $t$ ($\pi$ is a family of tangent lines to~$\D$ along~$\alpha$ 
which are all transverse to~$\ell$). Take the lift~$\check{\pi}: [0,1] \rightarrow \check{P}$ of~$\pi$ 
with~$\check{\pi}(0) = \check{c_0}(\alpha(0))$. The end points~$\check{\pi}(1)$ and $\check{c_1}(\alpha(1))$ 
are lifts of the same point in~$P$ so that there exists~$(\mu_1,\mu_2,\mu_3) \in \Z^3$ such that
$$\check{c_1}(\alpha(1)) = \chi(r_1,r_2,r_3) \check{\pi}(1).$$
As this triple of integers is independant of all the choices made, one can then define the local index 
of~$\check{c}_0, \check{c}_1$ at~$z$ as
$$\mu^{\tri} (\check{c}_0, \check{c}_1 ; z) =  (\mu_1,\mu_2,\mu_3) \in \Z^3.$$

The local index has an analog symmetry property as the one in~\cite{KhS}, namely:

\begin{lem}\label{symmetrylocalindex}
If $(c_0,\check{c}_0)$ and $(c_1,\check{c}_1)$ are two trigraded curves such that $c_0$ and $c_1$ have 
minimal intersection, then
$$\mu^{\tri} (\check{c}_1, \check{c}_0 ; z) = 
\left \{ { 
\begin{array}{cl}
(1,0,0)  -  \mu^{\tri} (\check{c}_0, \check{c}_1 ; z) & \mbox{ if } z \not \in \Delta  \\
(0,1,0)  -  \mu^{\tri} (\check{c}_0, \check{c}_1 ; z) & \mbox{ if } z \in \Delta \bs \{0 \}  \\
(0,0,1)  -  \mu^{\tri} (\check{c}_0, \check{c}_1 ; z) & \mbox{ if } z \in \{0\}. 
\end{array}
}\right.$$
\end{lem}

Let~$\check{c}_0$ and $\check{c}_1$ be two trigraded curves such that $c_0$ and $c_1$ have no intersection 
point in~$\partial \D$.
Take a curve~$c_1'$ isotopic to~$c_1$ with minimal intersection with~$c_0$. Then one can find a unique (by 
\ref{lem3.12-3.14}) trigrading~$\check{c}'_1$ of~$c'_1$ such that~$\check{c}'_1$ is isotopic to~$\check{c}_1$. 
Then the trigraded intersection number of~$\check{c}_0$ and $\check{c}_1$ is 
\begin{align*}
I^{\tri}(\check{c}_0,\check{c}_1) =& (1+q_1^{-1} q_2)
\left( {\sum_{z \in (c_0 \cap c_1')\bs \Delta} q_1^{\mu_1(z)+n \mu_3(z)}q_2^{\mu_2(z)-n \mu_3(z)}q_3^{-\mu_3(z)}}\right) \\
 &+ \sum_{z \in (c_0 \cap c_1')\cap \{1, \ldots, n\}} q_1^{\mu_1(z)+n \mu_3(z)}q_2^{\mu_2(z)-n \mu_3(z)}q_3^{-\mu_3(z)}\\
 &+ \frac{1}{2}(1+q_1^{-n} q_2^{n+1}q_3)
\sum_{z \in (c_0 \cap c_1')\cap \{0\}} q_1^{\mu_1(z)+n \mu_3(z)}q_2^{\mu_2(z)-n \mu_3(z)}q_3^{-\mu_3(z)}
\end{align*}
where $(\mu_1(z),\mu_2(z),\mu_3(z))=\mu^{\tri} (\check{c}_0, \check{c}_1 ; z).$
This trigraded intersection number is independent of the choice of~$c'_1$ and is an invariant of the 
isotopy classes of~$\check{c}_0$ and $\check{c}_1$. 
In the case $c_0$ and $c_1$ have intersection in~$\partial \D$, one uses a flow which moves~$\partial \D$ in the 
positive sense and define the trigraded intersection number as above, but this case will not occur in the sequel.

\begin{lem}\label{proprieteT}
The trigraded intersection number has the following properties:
\begin{itemize}
  \item[(T1)] $I(c_0,c_1) = \frac{1}{2} I^{\tri}(\check{c}_0,\check{c}_1)_{|q_1=q_2=q_3=1}$.
\item[(T2)] For any~$f \in \Diff(\D, \Delta, \{0\})$, 
$I^{\tri}(\check{f}(\check{c}_0),\check{f}(\check{c}_1)) = I^{\tri}(\check{c}_0,\check{c}_1) $.
\item[(T3)] For any~$(r_1,r_2,r_3) \in \Z^3$, 
\begin{align*}
 I^{\tri}(\check{c}_0,\chi(r_1,r_2,r_3)\check{c}_1) &=  I^{\tri}(\chi(-r_1,-r_2,-r_3)\check{c}_0,\check{c}_1) \\
&= q_1^{r_1+nr_3} q_2^{r_2-nr_3} q_3^{-r_3} I^{\tri}(\check{c}_0,\check{c}_1).
\end{align*}
\item[(T4)]  If $c_0 \cap c_1 \cap \partial \D = \emptyset$ and 
$I^{\tri}(\check{c}_0,\check{c}_1) = \sum_{r_1,r_2,r_3} a_{r_1,r_2,r_3} q_1^{r_1}q_2^{r_2} q_3^{r_3}$, then 
$I^{\tri}(\check{c}_1,\check{c}_0) = \sum_{r_1,r_2,r_3} a_{r_1,r_2,r_3} q_1^{-r_1}q_2^{1-r_2} q_3^{-r_3}$.
\end{itemize}
\end{lem}

The property (T4) is as in the bigraded case a consequence of Lemma~\ref{symmetrylocalindex}.

\subsection{Normal form}

Fix again the basic set of curves $b_1, \ldots, b_n$ as in Figure~\ref{fig:arcsbi} and now consider 
curves $d_1, \ldots d_n$ as in Figure~\ref{fig:arcsdi} which divide the disc $\D$ into regions $D_1, \ldots, D_n$.
\begin{figure}[htbp]
  \begin{center}
  \psfrag{0}{$\scriptscriptstyle{0}$}
  \psfrag{1}{$\scriptscriptstyle{1}$}
  \psfrag{2}{$\scriptscriptstyle{2}$}
  \psfrag{n}{$\scriptscriptstyle{n}$}
  \psfrag{i}{$\scriptscriptstyle{i}$}
  \psfrag{i+1}{$\scriptscriptstyle{i+1}$}
  \psfrag{bi}{$\scriptscriptstyle{b_i}$}
  \psfrag{di}{$\scriptscriptstyle{d_i}$}
  \psfrag{d1}{$\scriptscriptstyle{d_1}$}
  \psfrag{dn}{$\scriptscriptstyle{d_n}$}
  \psfrag{dn-1}{$\scriptscriptstyle{d_{n-1}}$}
  \psfrag{D1}{$\scriptscriptstyle{D_1}$}
  \psfrag{Dn}{$\scriptscriptstyle{D_{n}}$}
   \includegraphics[height=4.5cm]{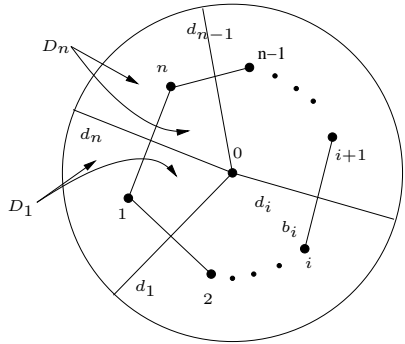} 
   \caption{The arcs $b_i$ and $d_i$ and the sectors $D_i$.}
   \label{fig:arcsdi}
  \end{center}
\end{figure}

It can be useful to depict precisely the effect on the sets of curves $b_k$ and $d_k$ of the chosen homeomorphism of the disk that goes from the affine configuration of Figure~\ref{fig:punctdisk2} to the aligned one of Figure~\ref{fig:punctdiskpi1}
if the reader wants to vizualize the lifts in $\check{P}$ using the description of the cover $\tilde{\D}_G$ of Section~\ref{DG}.
We give the images of these curves (still denoted $b_k$ and $d_k$) for a possible choice of homeomorphism in Figure~\ref{fig:affinealigne}.

\begin{figure}[htbp]
\begin{center}
\psfrag{D}{$\D$}
\psfrag{0}{$\scriptscriptstyle{0}$}
\psfrag{1}{$\scriptscriptstyle{1}$}
\psfrag{2}{$\scriptscriptstyle{2}$}
\psfrag{n-1}{$\scriptscriptstyle{n-1}$}
\psfrag{n}{$\scriptscriptstyle{n}$}
\psfrag{dots}{$\dots$}
\psfrag{b1}{$\scriptscriptstyle{b_1}$}
\psfrag{bn-1}{$\scriptscriptstyle{b_{n-1}}$}
\psfrag{bn}{$\scriptscriptstyle{b_n}$}
\psfrag{d1}{$\scriptscriptstyle{d_1}$}
\psfrag{dn-2}{$\scriptscriptstyle{d_{n-2}}$}
\psfrag{dn-1}{$\scriptscriptstyle{d_{n-1}}$}
\psfrag{dn}{$\scriptscriptstyle{d_n}$}
\includegraphics[height=4.5cm]{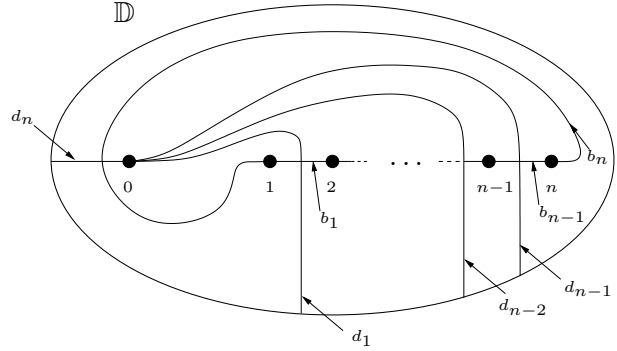}
\caption{The arcs $b_i$ and $d_i$ in the aligned configuration.}
\label{fig:affinealigne}
\end{center}
\end{figure}

We will say that an admissible curve~$c$ is in normal form if it has minimal intersection with all the~$d_i$. One 
can always achieve normal form by isotopy. The study of curves in this section makes sense because of the following uniqueness result:

\begin{lem} \label{lemma3.15}
Let $c_0$ and $c_1$ be two isotopic curves, both of which are in normal form. Then there is an isotopy 
relative to~$d_1 \cup d_2 \cup \ldots \cup d_n$ which carries~$c_0$ to~$c_1$. 
\end{lem}

Let~$c$ be a curve in normal form. Then each connected component of~$c \cap D_k$ belongs to one 
of the six following type depicted in Figure~\ref{fig:6types}.
\begin{figure}[htbp]
  \begin{center}
  \psfrag{0}{$\scriptscriptstyle{0}$}
  \psfrag{k}{$\scriptscriptstyle{k}$}
  \psfrag{c}{$\scriptscriptstyle{c}$}
  \psfrag{dk}{$\scriptscriptstyle{d_k}$}
  \psfrag{dk-1}{$\scriptscriptstyle{d_{k-1}}$}
  \psfrag{Dk}{$\scriptscriptstyle{D_{k}}$}
  \psfrag{Type 1}{\scriptsize{Type 1}}
  \psfrag{Type 1'}{\scriptsize{Type 1'}}
  \psfrag{Type 2}{\scriptsize{Type 2}}
  \psfrag{Type 2'}{\scriptsize{Type 2'}}
  \psfrag{Type 3}{\scriptsize{Type 3}}
  \psfrag{Type 3'}{\scriptsize{Type 3'}}
   \includegraphics[height=10cm]{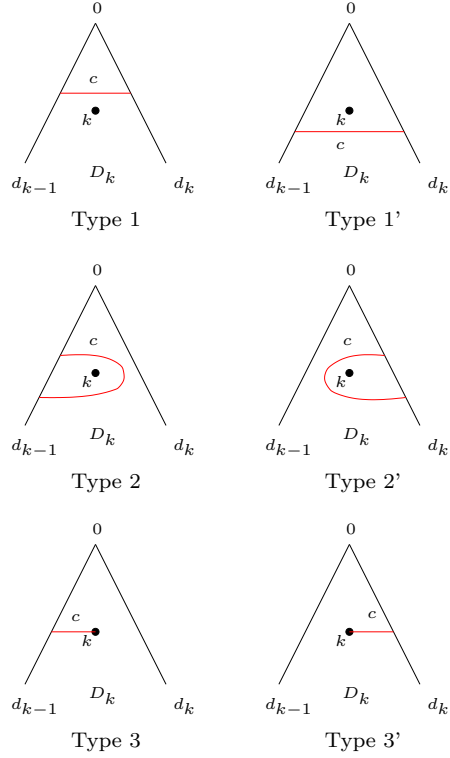} 
   \caption{The six possible types.}
   \label{fig:6types}
  \end{center}
\end{figure}

Conversely, an admissible curve~$c$ which intersects all the~$d_k$ transversally and such that each 
connected component of~$c \cap D_k$ belongs to one of the types listed in the Figure~\ref{fig:6types}  
is already in normal form.

To be more precise, one should actually work with the group of diffeomorphisms of~$D_k$ which 
fix~$\partial \D \cap D_k \cup \{0\}$ and preserve~$d_{k-1}$, $d_{k}$ and $k$. 
The list of types 
then classifies each connected component of~$c\cap D_k$ up to an isotopy in this group. 
Lemma~\ref{lemma3.15} then shows that the connected components of each type, their relative position, 
and the way they join each other, is an invariant of the isotopy class of $c$.

For the rest of this section, $c$ is an admissible curve in normal form.
We will call crossing and denote~$cr(c) = c \cap (d_1 \cup d_2 \cup \ldots \cup d_n)$ the intersections of 
this curve with the barriers of the sectors and the intersections with~$d_k$ will be called $k$-crossings of~$c$.
The connected components of~$c \cap D_k, 1 \leq k \leq n$ are called segments of~$c$, and a segment is 
said essential if its endpoints are both crossings (and not punctures). 
So, the essential segments are the ones of type 1, 1', 2, 2', and the basic curves have no essential segments.

The curve~$c$ can be reconstructed up to isotopy by listing its crossings and the types of essential segments bounded by consecutive crossings as one travels along~$c$ from one end to another, and 
Lemma~\ref{lemma3.15} shows that conversely this combinatorial data is an invariant of the isotopy class 
of~$c$.\\

Let us now study the action of half-twists on normal forms.
Remember that we denoted~$t_{b_k}$ the half-twist along~$b_k$.
Even when the curve~$c$ is in normal form, the curve~$t_{b_k}(c)$ is not necessarly in normal form too. 
This image~$t_{b_k}(c)$ has minimal intersection with the~$d_i$ for~$i \neq k$ but one might need to simplify its intersections 
with~$d_k$ to get~$t_{b_k}(c)$ into normal form. The same argument as in~\cite{KhS} leads to the analogous 
statement:

\begin{prop} \label{Proposition3.17} 
\hfill
\begin{itemize}
\item[(i)] The normal form of~$t_{b_k}(c)$ coincides with~$c$ outside of~$D_k \cup D_{k+1}$. 
The curve~$t_{b_k}(c)$ can be brought into normal form by an isotopy inside~$D_k \cup D_{k+1}$. 
\item[(ii)] Assume~$t_{b_k}(c)$ is in normal form. There is a natural bijection between the $i$-crossings of~$c$ 
and the $i$-crossings of~$t_{b_k}(c)$ for $i \neq k$. There is a natural bijection between connected components 
of intersections of~$c$ and~$t_{b_k}(c)$ inside~$D_k \cup D_{k+1}$. 
\end{itemize}
\end{prop}

We need thus to study~$c$ and~$t_{b_k}(c)$ in~$D_k \cup D_{k+1}$. Just as before we will define objects for this 
situation. 
We call $k$-string of~$c$ a connected component of $c \cap (D_k \cup D_{k+1})$ and denote by $st(c,k)$ 
the set of $k$-strings of~$c$.
We call $k$-string a curve in~$D_k \cup D_{k+1}$ which is a $k$-string of~$c$ for some admissible 
curve~$c$ in normal form.
Two $k$-strings are isotopic if there is a deformation of one into the other via diffeomorphisms 
of~$D' = D_k \cup D_{k+1}$ which fix $d_{k-1}$ and $d_{k+1}$ and preserve the punctures 
lying inside of~$D'$ setwise. The isotopy classes under this equivalence relation can be divided into five families of 
types denoted $I_u, II_u, II'_u, III_u, III'_u$ ($u \in \Z$), which are depicted on Figure~\ref{fig:kstringk=2n-1}, in analogy to~\cite[Figure 15]{KhS}, for the $u=0$ case and such that the type~$u+1$ is obtained from the $u$-type by applying~$t_{b_k}$.
The $k$-strings are by definition in normal form and, as before, we can define for them crossings and essential segments and denote the crossings of a $k$-string $g$ similarly by~$cr(g)$.

For the geometric intersection numbers, the same argument as in~\cite{KhS} leads to the analogous result: 

\begin{prop} \label{Lemma 3.18}
If $k \in \{1, \dots, n\}$ the geometric intersection number $I(b_k,c)$ can be computed as follows: 
every $k$-string of~$c$ which is of type $I_u, II_u, II'_u$ contributes~$1$, 
those of type $III_u, III'_u$ contribute~$\frac{1}{2}$, and the other types contribute to~$0$.
\end{prop}

For trigraded curves, one can choose trigradings~$\check{b}_k, \check{d}_k$ of~$b_k, d_k$ such that
\begin{align*}
I^{\tri}(\check{d}_k,\check{b}_k) & = 1 + q_1^{-1} q_2\\ 
I^{\tri}(\check{b}_{k+1},\check{b}_k) & = 1 \qquad \text{for } i = 1, \dots n-1, \\
I^{\tri}(\check{b}_{1},\check{b}_n) & = q_3^{-1}
\end{align*}
\noindent
that is such that the local intersection indices at the intersection points are:
\begin{align*}
\mu^{\tri}(\check{d}_k,\check{b}_k;z) & = (0,0,0)\\ 
\mu^{\tri}(\check{b}_{k+1},\check{b}_k;z) & = (0,0,0) \qquad \text{for } i = 1, \dots n-1, \\
\mu^{\tri}(\check{b}_{1},\check{b}_n;z) & = (-n,n,1).
\end{align*}

These conditions determine the trigradings uniquely up to an overall shift~$\chi(r_1,r_2,r_3)$.

Now, if~$\check{c}$ is a trigrading of an admissible curve~$c$ in normal form and if~$a$ is a connected 
component of~$c \cap D_k$ for some $k$ and $\check{a}$ is the preimage in $\check{c}$ of~$a$ under 
the covering projection, then~$\check{a}$ is entirely determined by~$a$ and the local 
index~$\mu^{\tri}(\check{d}_{k-1},\check{a};z)$ or~$\mu^{\tri}(\check{d}_k,\check{a};z)$ at any 
point~$z \in (d_{k-1} \cup d_{k})\cap a$ (if there is more than one such point, the local indices determine each other).

In Figures~\ref{fig:6typesk=2n} and~\ref{fig:6typesk=1} we give the classification of the pairs~$(a,\check{a})$ for the various possible types with 
the local indices. Note that Figure~\ref{fig:6typesk=1} shows the only two types $1$ and $1'$ for which the values of the local indices differ in the case when $k=1$ from the generic case $k \neq1$. The indices mentioned have to be read the following way: if the connected 
component~$(a,\check{a})$ is of type, for example $1(r_1,r_2,r_3)$ for $k=2, \ldots,n$, with 
crossings~$z_0 \in d_{k} \cap a$ 
and~$z_1 \in d_{k-1} \cap a$ and the local index at~$z_0$ 
is~$\mu^{\tri}(\check{d}_{k},\check{a};z_0) = (r_1,r_2,r_3)$, 
then~$\mu^{\tri}(\check{d}_{k-1},\check{a};z_1) = (r_1-1,r_2+1,r_3)$.

\begin{figure}[htbp]
  \begin{center}
  \psfrag{0}{$\scriptscriptstyle{0}$}
  \psfrag{k}{$\scriptscriptstyle{k}$}
  \psfrag{c}{$\scriptscriptstyle{a}$}
  \psfrag{dk}{$\scriptscriptstyle{d_k}$}
  \psfrag{dk-1}{$\scriptscriptstyle{d_{k-1}}$}
  \psfrag{Dk}{$\scriptscriptstyle{D_{k}}$}
  \psfrag{Type 1}{\scriptsize{Type $1(r_1,r_2,r_3)$}}
  \psfrag{Type 1'}{\scriptsize{Type $1'(r_1,r_2,r_3)$}}
  \psfrag{Type 2}{\scriptsize{Type $2(r_1,r_2,r_3)$}}
  \psfrag{Type 2'}{\scriptsize{Type $2'(r_1,r_2,r_3)$}}
  \psfrag{Type 3}{\scriptsize{Type $3(r_1,r_2,r_3)$}}
  \psfrag{Type 3'}{\scriptsize{Type $3'(r_1,r_2,r_3)$}}
  \psfrag{r1r2r3}{$\scriptscriptstyle{(r_1,r_2,r_3)}$}
  \psfrag{r1-1r2+1r3}{$\scriptscriptstyle{(r_1-1,r_2+1,r_3)}$}
  \psfrag{r1-1r2r3}{$\scriptscriptstyle{(r_1-1,r_2,r_3)}$}
  \psfrag{r1+1r2-1r3}{$\scriptscriptstyle{(r_1+1,r_2-1,r_3)}$}
   \includegraphics[height=10cm]{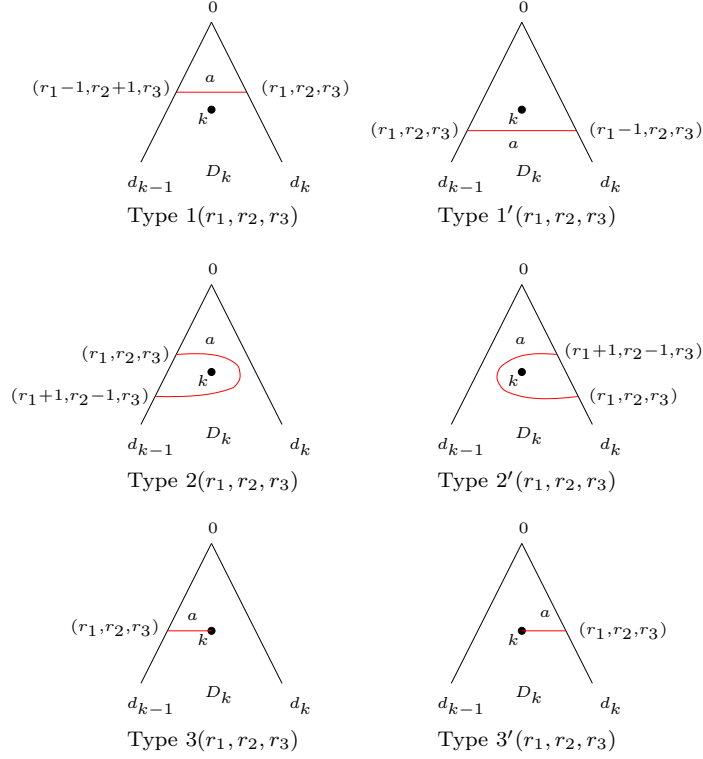} 
   \caption{The local indices for $k=2, \ldots, n$.}
   \label{fig:6typesk=2n}
  \end{center}
\end{figure}

\begin{figure}[htbp]
  \begin{center}
  \psfrag{0}{$\scriptscriptstyle{0}$}
  \psfrag{k}{$\scriptscriptstyle{1}$}
  \psfrag{c}{$\scriptscriptstyle{a}$}
  \psfrag{dk}{$\scriptscriptstyle{d_1}$}
  \psfrag{dk-1}{$\scriptscriptstyle{d_{n}}$}
  \psfrag{Dk}{$\scriptscriptstyle{D_{1}}$}
  \psfrag{Type 1}{\scriptsize{Type $1(r_1,r_2,r_3)$}}
  \psfrag{Type 1'}{\scriptsize{Type $1'(r_1,r_2,r_3)$}}
  \psfrag{Type 2}{\scriptsize{Type $2(r_1,r_2,r_3)$}}
  \psfrag{Type 2'}{\scriptsize{Type $2'(r_1,r_2,r_3)$}}
  \psfrag{Type 3}{\scriptsize{Type $3(r_1,r_2,r_3)$}}
  \psfrag{Type 3'}{\scriptsize{Type $3'(r_1,r_2,r_3)$}}
  \psfrag{r1r2r3}{$\scriptscriptstyle{(r_1,r_2,r_3)}$}
  \psfrag{r1-1r2+1r3}{$\scriptscriptstyle{(r_1+n-1,r_2-n+1,r_3-1)}$}
  \psfrag{r1-1r2r3}{$\scriptscriptstyle{(r_1-n-1,r_2+n,r_3+1)}$}
  \psfrag{r1+1r2-1r3}{$\scriptscriptstyle{(r_1+1,r_2-1,r_3)}$}
   \includegraphics[height=3cm]{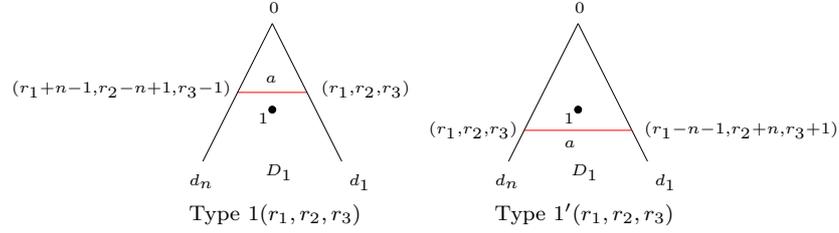} 
   \caption{The change of local indices for $k=1$.}
   \label{fig:6typesk=1}
  \end{center}
\end{figure}

Remember that we defined, at the beginning of Section~\ref{PTD}, for any 
diffeomorphism~$f \in \Diff(\D, \Delta, \{0\})$ its prefered lift~$\check{f}$ as the only lift that acts 
trivially on the preimage of any point of the boundary of the disk. Remember also that we considered in 
Section~\ref{BraidMCG} the half-twist $t_{b_i}$ along the curve~$b_i$ and the twist $t_{\partial}$ corresponding 
to the generator~$\rho$ of the braid group, so that in the 
following~$\check{t}_{b_i}$ and $\check{t}_{\partial}$ will be the prefered lift of these diffeomorphisms.

\begin{prop} \label{Prop3.19}
The diffeomorphisms~$\check{t}_{b_i}$, $i = 1 \ldots n$ and $\check{t}_{\partial}$ induce an extended 
affine type A braid group action on the isotopy classes of admissible trigraded curves: if~$\check{c}$ is 
an admissible trigraded curve, we have the following isotopy relations of curves:
\begin{align*}
\check{t}_{b_i} \check{t}_{b_j} (\check{c}) & = \check{t}_{b_j} \check{t}_{b_i} (\check{c}) 
& \mbox{for distant} \ i,j=1,\dots, n, \\
\check{t}_{b_i} \check{t}_{b_{i+1}} \check{t}_{b_i}  (\check{c}) & =  \check{t}_{b_{i+1}} \check{t}_{b_i}   \check{t}_{b_{i+1}}(\check{c}) 
& \mbox{for} \ i=1,\dots, n, \\ 
\check{t}_{\partial} \check{t}_{b_i} \check{t}_{\partial}^{-1} & = \check{t}_{b_{i+1}} 
& \mbox{for} \ i=1,\dots, n.
\end{align*}
\end{prop}

A crossing of~$c$ is also called a crossing of~$\check{c}$ and we denote also~$cr(\check{c})$ the set of 
crossing of~$\check{c}$,  $cr(\check{c}) = cr(c)$.
But to crossings of~$\check{c}$, we can also associate its local index in~$\Z^3$. 
As~\cite{KhS}, we add an extra index~$k$ to a crossing $z$ which is equal to the index~$k$ such that~$z \in d_k \cap c$ 
and to 
emphazise that it is a function of the intersection point we will denote it $k(z)$ such that  we get a map 
which associate to each crossing~$z$ the four intergers~$(k(z), \mu_1(z), \mu_2(z), \mu_3(z))$.

We call essential segments of~$\check{c}$ the essential segments of~$c$ together with the trigrading, 
the trigradings are given by assigning local indices to the ends of the segment.
We do now the same study about the changes of trigradings for $k$-strings. 
We call $k$-string of~$\check{c}$ a connected component of~$\check{c} \cap (D_k \cup D_{k+1})$ 
together with the trigrading induced by the one of~$\check{c}$.
The set of $k$-strings of~$\check{c}$ are denoted~$st(\check{c},k)$.
A trigraded $k$-string is a $k$-string of~$\check{c}$ for some trigraded curve~$\check{c}$.

In Figures~\ref{fig:kstringk=2n-1}, \ref{fig:kstringk=1} and \ref{fig:kstringk=n}, we give the isotopy classes of bigraded $k$-strings. As before, are appart on Figures\ref{fig:kstringk=1} and \ref{fig:kstringk=n} the only types for which the values of the local indices differ in the case when $k=1,n$ from the generic case $k \neq1,n$.
\begin{center}
\begin{figure}[htbp]
  \psfrag{0}{$\scriptscriptstyle{0}$}
  \psfrag{k}{$\scriptscriptstyle{k}$}
  \psfrag{k+1}{$\scriptscriptstyle{k+1}$}
  \psfrag{c}{$\scriptscriptstyle{c}$}
  \psfrag{dk+1}{$\scriptscriptstyle{d_{k+1}}$}
  \psfrag{dk-1}{$\scriptscriptstyle{d_{k-1}}$}
  \psfrag{r1r2r3}{$\scriptscriptstyle{(r_1,r_2,r_3)}$}
  \psfrag{r1r2-1r3}{$\scriptscriptstyle{(r_1,r_2-1,r_3)}$}
  \psfrag{r1+1r2-1r3}{$\scriptscriptstyle{(r_1+1,r_2-1,r_3)}$}
  \psfrag{r1+2r2-2r3}{$\scriptscriptstyle{(r_1+2,r_2-2,r_3)}$}
  \psfrag{r1+2r2r3}{$\scriptscriptstyle{(r_1+2,r_2,r_3)}$}
  \psfrag{r1+2r2r3}{$\scriptscriptstyle{(r_1+2,r_2,r_3)}$}
  \psfrag{r1+3r2-2r3}{$\scriptscriptstyle{(r_1+3,r_2-2,r_3)}$}
 \psfrag{r1-3r2+2r3}{$\scriptscriptstyle{(r_1-3,r_2+2,r_3)}$}
  \psfrag{Type I_0(r_1,r_2,r_3)}{$\scriptscriptstyle{\text{Type } I_0 (r_1,r_2,r_3)}$}
  \psfrag{Type II_0(r_1,r_2,r_3)}{$\scriptscriptstyle{\text{Type } II_0 (r_1,r_2,r_3)}$}
  \psfrag{Type II'_0(r_1,r_2,r_3)}{$\scriptscriptstyle{\text{Type } II'_0 (r_1,r_2,r_3)}$}
  \psfrag{Type III_0(r_1,r_2,r_3)}{$\scriptscriptstyle{\text{Type } III_0 (r_1,r_2,r_3)}$}
  \psfrag{Type III'_0(r_1,r_2,r_3)}{$\scriptscriptstyle{\text{Type } III'_0 (r_1,r_2,r_3)}$}
  \psfrag{Type IV(r_1,r_2,r_3)}{$\scriptscriptstyle{\text{Type } IV (r_1,r_2,r_3)}$}
  \psfrag{Type IV'(r_1,r_2,r_3)}{$\scriptscriptstyle{\text{Type } IV' (r_1,r_2,r_3)}$}  
  \psfrag{Type V(r_1,r_2,r_3)}{$\scriptscriptstyle{\text{Type } V (r_1,r_2,r_3)}$}
  \psfrag{Type V'(r_1,r_2,r_3)}{$\scriptscriptstyle{\text{Type } V' (r_1,r_2,r_3)}$}
  \psfrag{Type VI(r_1,r_2,r_3)}{$\scriptscriptstyle{\text{Type } VI (r_1,r_2,r_3)}$}
\begin{tabular}{ccc}
  \includegraphics[height=3.5cm]{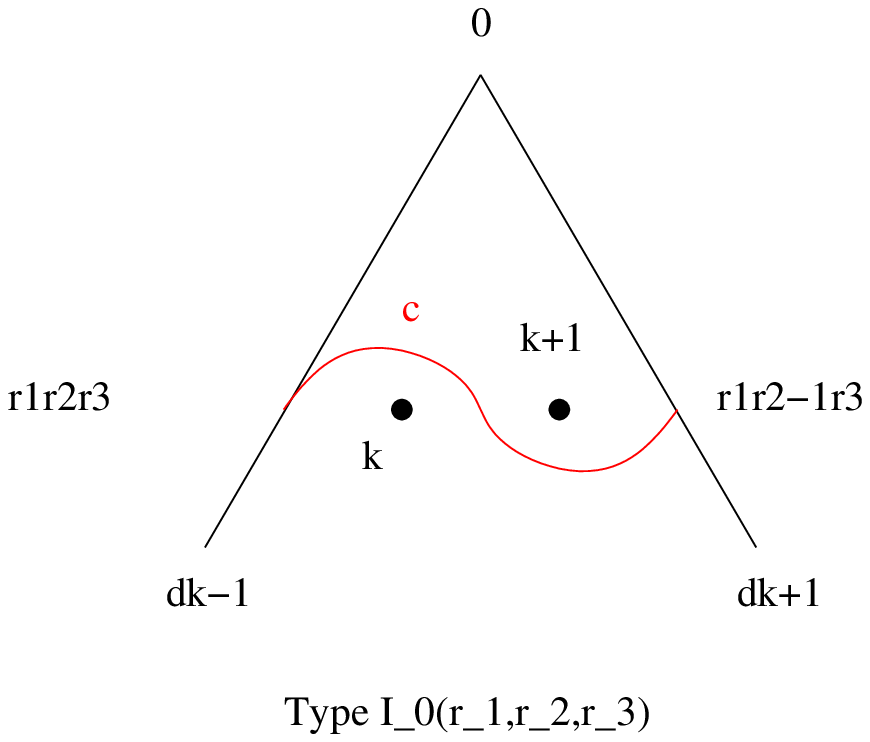}  & \hspace{0.5cm} &   
\includegraphics[height=3.5cm]{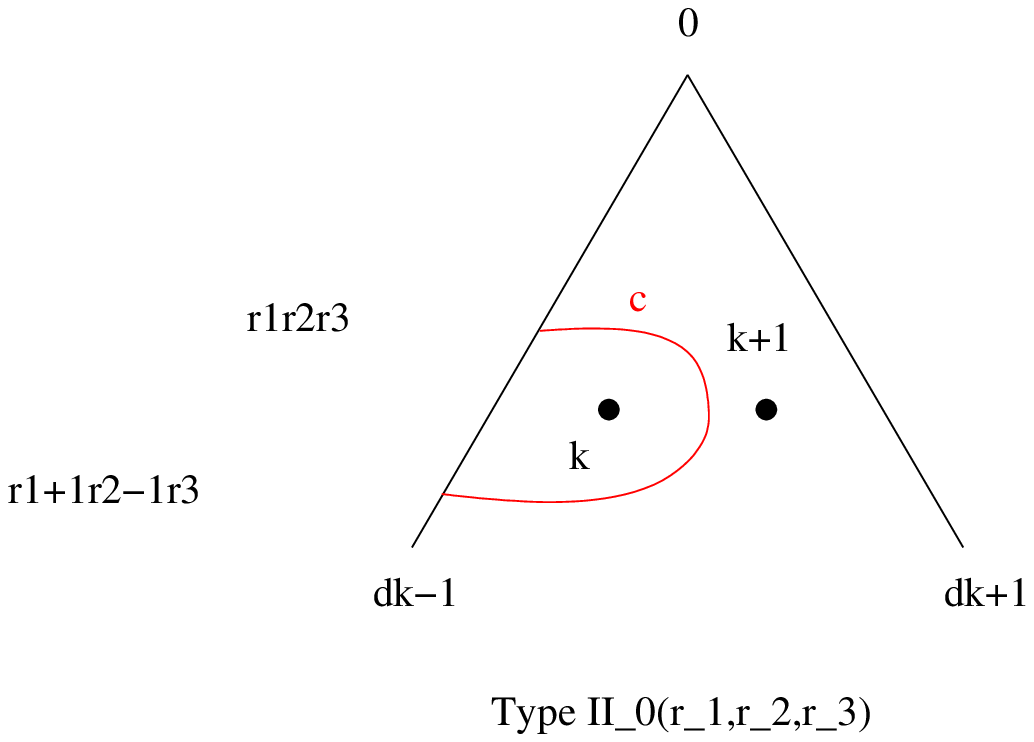} \\
 \includegraphics[height=3.5cm]{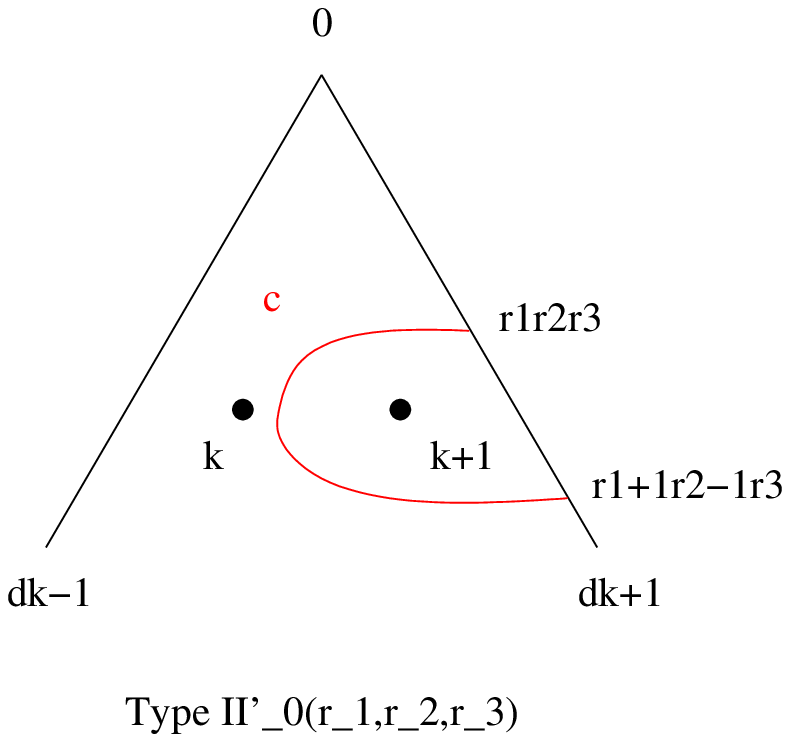}  & \hspace{0.5cm} &   \includegraphics[height=3.5cm]{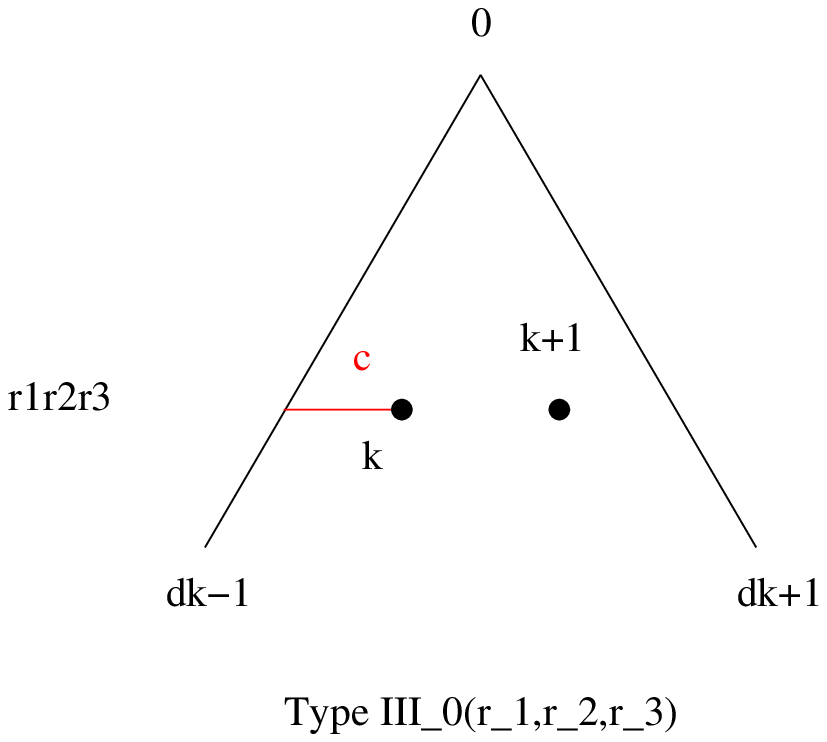} \\
\includegraphics[height=3.5cm]{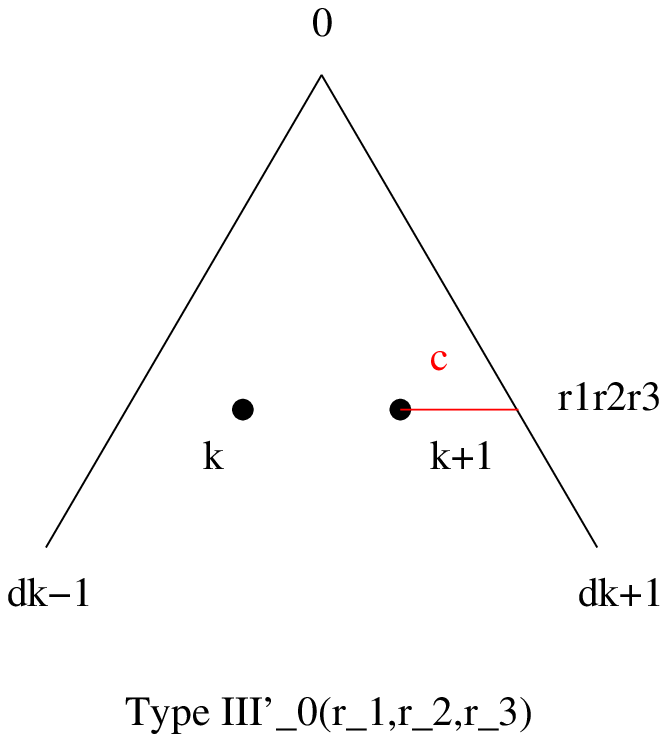}  & \hspace{0.5cm} &  
 \includegraphics[height=3.5cm]{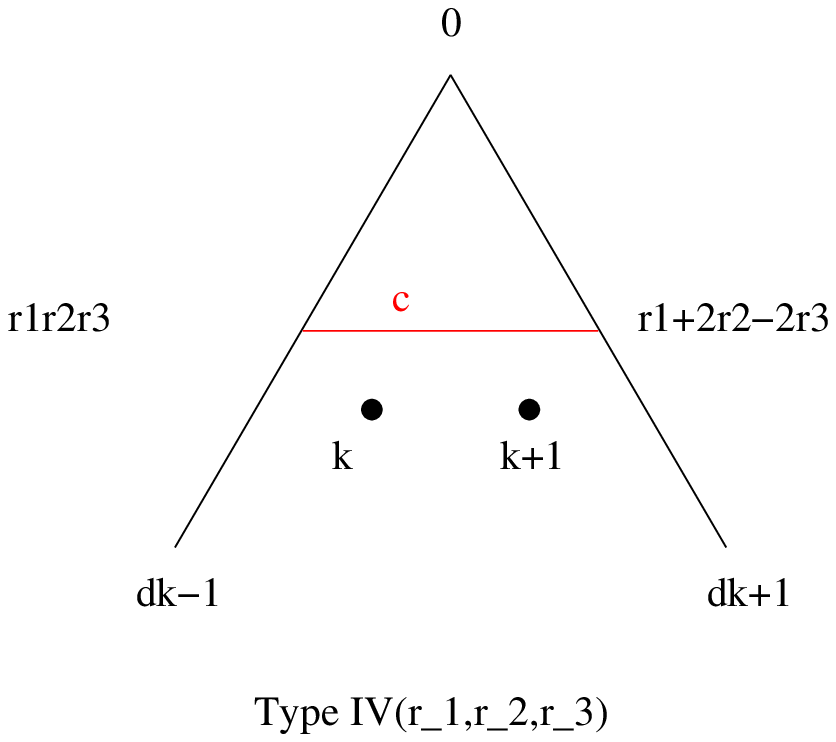} \\
\includegraphics[height=3.5cm]{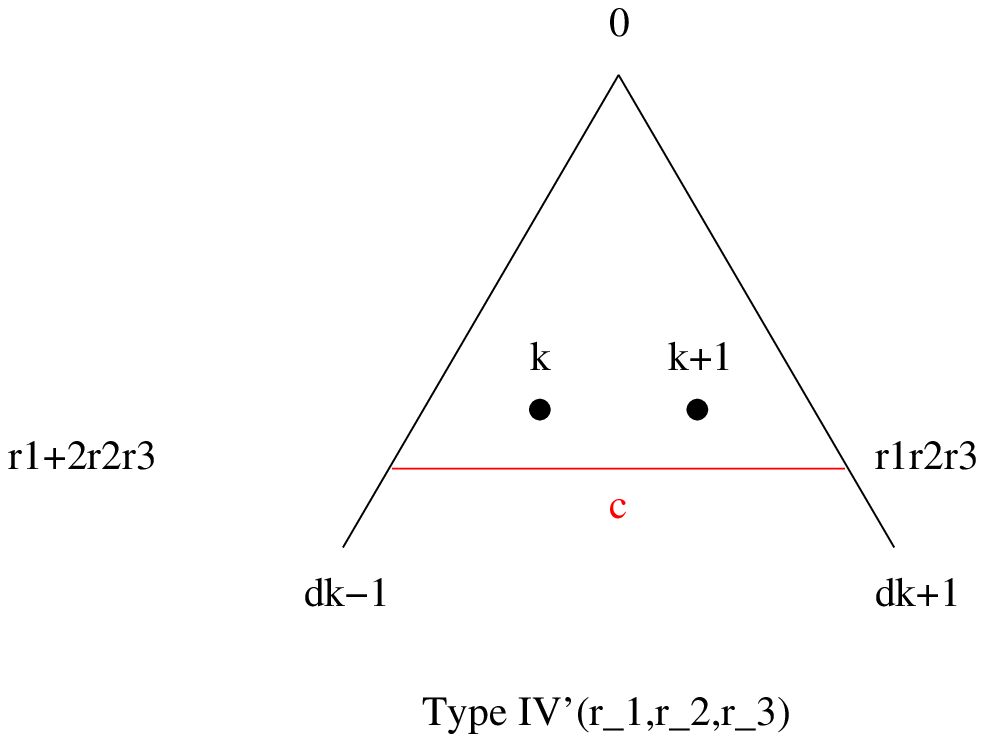}  & \hspace{0.5cm} &   
\includegraphics[height=3.5cm]{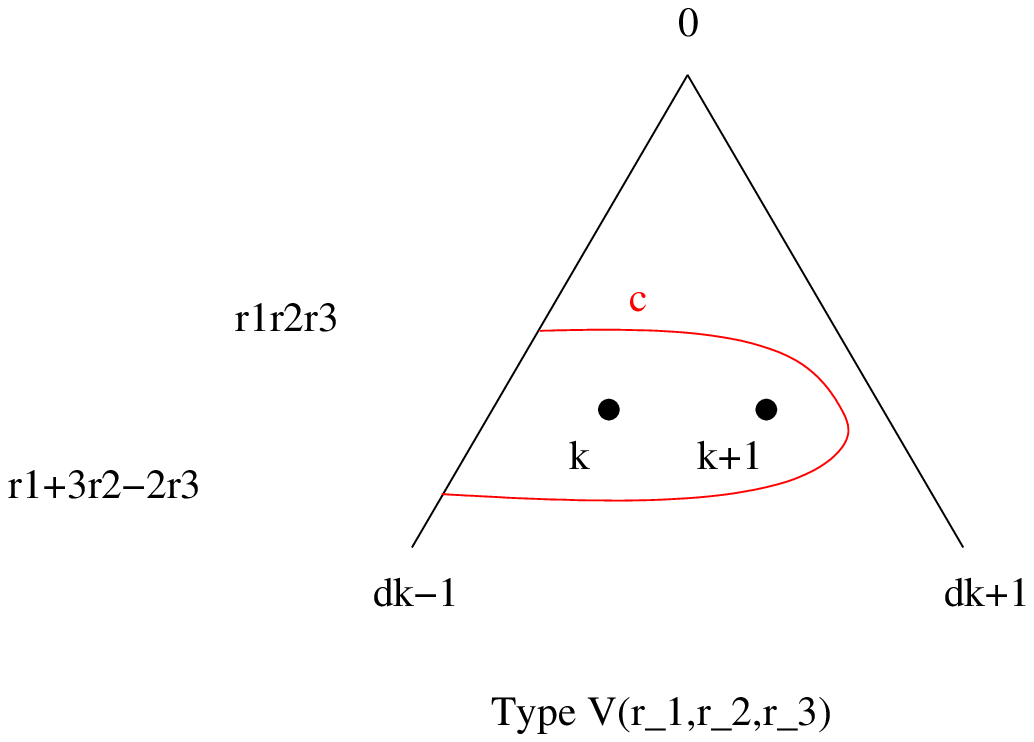} \\
\includegraphics[height=3.5cm]{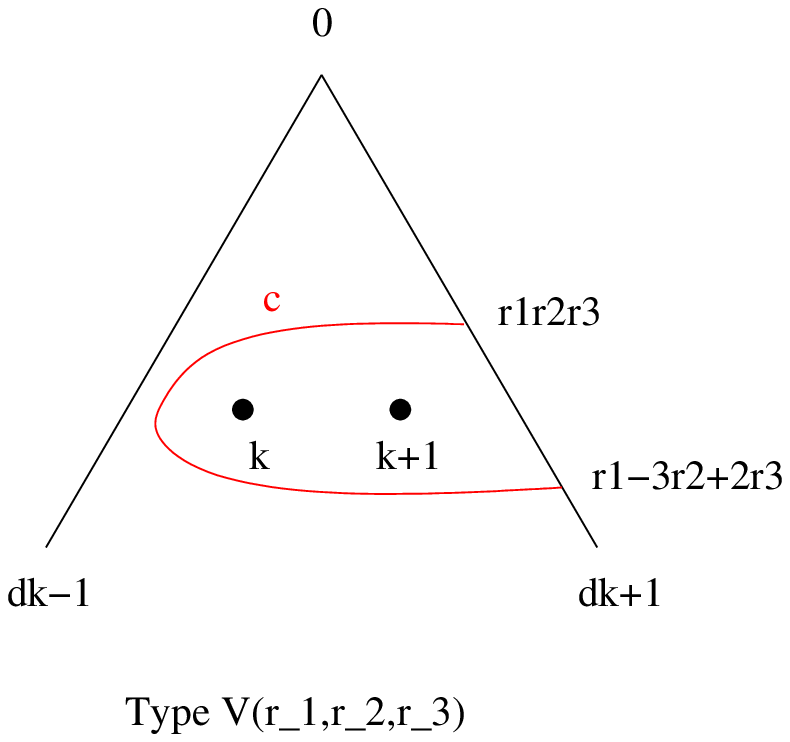}  & \hspace{0.5cm} &   
\includegraphics[height=3.5cm]{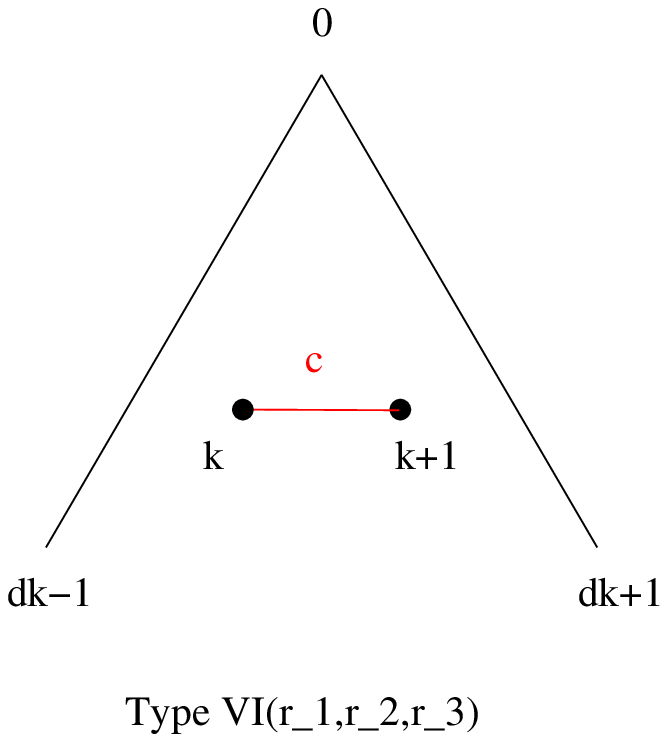} \\
\end{tabular}
   \caption{The isotopy classes of $k$-strings with local indices for $k=2, \ldots, n-1$.}
   \label{fig:kstringk=2n-1}
\end{figure}
  \end{center}

\begin{center}
\begin{figure}[htbp]
  \psfrag{0}{$\scriptscriptstyle{0}$}
  \psfrag{k}{$\scriptscriptstyle{1}$}
  \psfrag{k+1}{$\scriptscriptstyle{2}$}
  \psfrag{c}{$\scriptscriptstyle{c}$}
  \psfrag{dk+1}{$\scriptscriptstyle{d_{2}}$}
  \psfrag{dk-1}{$\scriptscriptstyle{d_{n}}$}
  \psfrag{r1r2r3}{$\scriptscriptstyle{(r_1,r_2,r_3)}$}
  \psfrag{r1r2-1r3}{$\scriptscriptstyle{(r_1-n,r_2+n-1,r_3+1)}$}
  \psfrag{r1+2r2-2r3}{$\scriptscriptstyle{(r_1-n+2,r_2+n-2,r_3+1)}$}
  \psfrag{r1+2r2r3}{$\scriptscriptstyle{(r_1+n+2,r_2-n,r_3-1)}$}
  \psfrag{Type I_0(r_1,r_2,r_3)}{$\scriptscriptstyle{\text{Type } I_0 (r_1,r_2,r_3)}$}
  \psfrag{Type IV(r_1,r_2,r_3)}{$\scriptscriptstyle{\text{Type } IV (r_1,r_2,r_3)}$}
  \psfrag{Type IV'(r_1,r_2,r_3)}{$\scriptscriptstyle{\text{Type } IV' (r_1,r_2,r_3)}$}  
\begin{tabular}{ccc}
  \includegraphics[height=3.5cm]{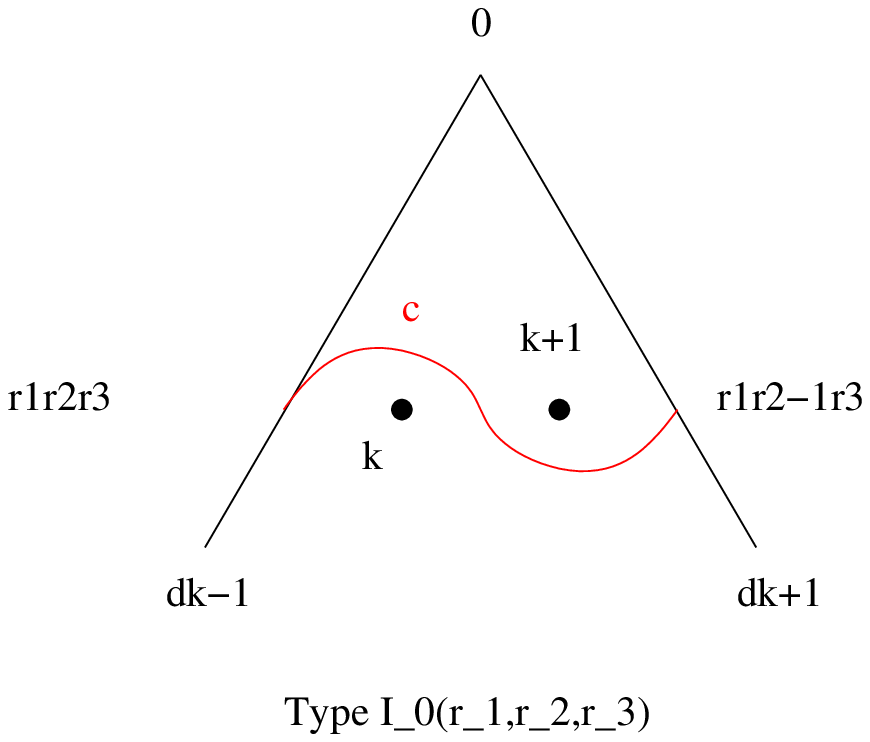}  & \hspace{0.5cm} &   
  \includegraphics[height=3.5cm]{Fig15k=2n-1IV.eps}  \\
  \includegraphics[height=3.5cm]{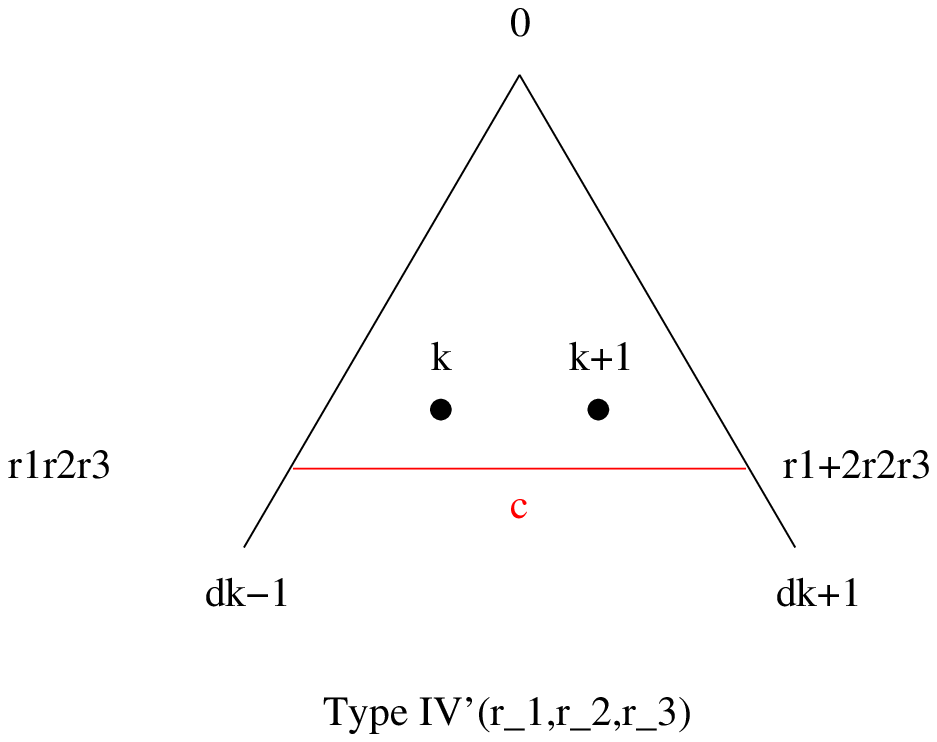}  
\end{tabular}
   \caption{The changes in the isotopy classes of $k$-strings with local indices when $k=1$.}
   \label{fig:kstringk=1}
\end{figure}
  \end{center}

\begin{center}
\begin{figure}[htbp]
  \psfrag{0}{$\scriptscriptstyle{0}$}
  \psfrag{k}{$\scriptscriptstyle{n}$}
  \psfrag{k+1}{$\scriptscriptstyle{1}$}
  \psfrag{c}{$\scriptscriptstyle{c}$}
  \psfrag{dk+1}{$\scriptscriptstyle{d_{1}}$}
  \psfrag{dk-1}{$\scriptscriptstyle{d_{n-1}}$}
  \psfrag{r1r2r3}{$\scriptscriptstyle{(r_1,r_2,r_3)}$}
  \psfrag{r1r2-1r3}{$\scriptscriptstyle{(r_1-n,r_2+n-1,r_3+1)}$}
  \psfrag{r1+2r2-2r3}{$\scriptscriptstyle{(r_1-n+2,r_2+n-2,r_3+1)}$}
  \psfrag{r1+2r2r3}{$\scriptscriptstyle{(r_1-n-2,r_2+n,r_3+1)}$}
  \psfrag{Type I_0(r_1,r_2,r_3)}{$\scriptscriptstyle{\text{Type } I_0 (r_1,r_2,r_3)}$}
  \psfrag{Type IV(r_1,r_2,r_3)}{$\scriptscriptstyle{\text{Type } IV (r_1,r_2,r_3)}$}
  \psfrag{Type IV'(r_1,r_2,r_3)}{$\scriptscriptstyle{\text{Type } IV' (r_1,r_2,r_3)}$}  
\begin{tabular}{ccc}
  \includegraphics[height=3.5cm]{Fig15k=1I0.eps}  & \hspace{0.5cm} &   
  \includegraphics[height=3.5cm]{Fig15k=2n-1IV.eps}  \\
  \includegraphics[height=3.5cm]{Fig15k=1IVprim.eps}  
\end{tabular}
   \caption{The changes in the isotopy classes of $k$-strings with local indices when $k=n$.}
   \label{fig:kstringk=n}
\end{figure}
  \end{center}

The trigraded intersection number of~$\check{b}_k$ with a trigraded curve can be computed 
thanks to the following lemma:

\begin{lem} \label{Lemma 3.20}
Let~$(c,\check{c})$ be a trigraded curve.
Then~$I^{\tri}(\check{b}_k,c)$ can be computed by adding up contributions from each trigraded 
$k$-string of~$\check{c}$.
For~$k=2, \ldots, n-1$ the contributions are:\\
%
\noindent
\begin{tabular}{|c|c|}
\hline
$I_0(0,0,0)$ & $1+q_1 q_2^{-1}$\\
\hline
$II_0(0,0,0)$ & $1+q_1 q_2^{-1}$ \\
\hline
$II'_0(0,0,0)$ & $q_1+q_2$  \\
\hline
$III_0(0,0,0)$ & $1$ \\
\hline
$III'_0(0,0,0)$ & $q_2$  \\
\hline
\end{tabular}
\hspace{1cm}
\begin{tabular}{|c|c|}
\hline
$IV(0,0,0)$ & $0$ \\
\hline
$IV'(0,0,0)$ & $0$ \\
\hline
$V(0,0,0) $ & $0$ \\
\hline
$V'(0,0,0)$ & $0$ \\
\hline
$VI(0,0,0)$ &  $1 + q_2$\\
\hline 
\end{tabular}

\noindent
For~$k=1$ the only contributions that differ are:\\
\begin{tabular}{|c|c|}
\hline
$I_0(0,0,0)$ & $(1+q_1 q_2^{-1})q_3^{-1}$\\
\hline
$II_0(0,0,0)$ & $(1+q_1 q_2^{-1})q_3^{-1}$\\
\hline
$III_0(0,0,0)$  & $q_3^{-1}$ \\
\hline
\end{tabular}

\noindent
For~$k=n$ the only contributions that differ are:\\
\begin{tabular}{|c|c|}
\hline
 $II'_0(0,0,0)$ & $(q_1 + q_2)q_3$\\
\hline
$III_0(0,0,0)$ & $q_2q_3$ \\
\hline
\end{tabular}

\noindent
and for~$u \neq 0$ the contribution of type~$X_u(r_1,r_2,r_3)$ is 
$q_1^{r_1 + n r_3}q_2^{r_1- n r_3} q_3^{-r_3} (q_1 q_2^{-1})^{u}$ times the 
contribution of~$X_0(0,0,0)$, with 
$$X \in \{I,II,II',III,III',IV,IV',V,V',VI\}.$$
\end{lem}

\section{Admissible curves and complexes of projective modules}


In this section, we make the link between the categorical action and the geometry and topology of the trigraded intersection numbers. We use this link to prove the faithfulness of the categorical action defined in Section \ref{action}. In addition let us make a warranty here. In fact in Section \ref{action} we introduced a $\mathbb{Z}/2\mathbb{Z}$-grading (the reduction of the path length grading) on the category ${\cat}_n$ in order to recover exactly the homological representation defined in Section~\ref{rephomologique} at the level of the Grothendieck group. In all what follows, we do not need this grading and hence will work over the homotopy category of finitely generated bigraded projective modules that we denote ${\cat}'_n$. Note that the functors $\F_{\sigma}$ introduced in Section \ref{action} induce naturally endofunctors of ${\cat}'_n$, and we will therefore use the same notation.

\subsection{Complexes associated to admissible curves and categorical action}

Given an admissible curve $\check{c}$ in normal form, we associate an object $L(\check{c})$ in the category ${\cat}'_n$ of bounded complexes of projective bigraded modules over the quiver algebra ${R}_n$.
We define $L(\check{c})$ first as a trigraded ${R}_n$-module as follows:
$$L(\check{c})=\bigoplus_{z\in cr(\check{c})} P(z),$$
where $P(z)=P_{k(z)} [-\mu_1(z)-n\mu_3(z)]\{\mu_2(z)-n\mu_3(z)\}\left<-\mu_3(z)\right>$ with $[-]$ being a shift in the cohomological grading. We endow the previous trigraded module with a differential given by:
\begin{itemize}
\item If $z_0$ and $z_1$ are two boundary crossings of an essential segment then it follows that they differ in their $\mu_1$ grading by $1$. Suppose for instance that $\mu_1(z_1)=\mu_1(z_0)+1$. There are two possibilities:
\begin{itemize}
\item If $z_0$ and $z_1$ are both $k$ crossings then $\partial :P(z_0)\rightarrow P(z_1)$ is the right multiplication by the element $(k|k+1|k)$.
\item If $z_0$ (resp. $z_1$) is a $k_0$-crossing  (resp. a $k_1$-crossing) and we have $|k_0-k_1|=1$, then $\partial :P(z_0)\rightarrow P(z_1)$ is the right multiplication by the element $(k_0|k_1)$.
\end{itemize}
\item If $z_0$ and $z_1$ are not connected by an essential segment then there is no contribution of the differential between $P(z_0)$ and $P(z_1)$.
\end{itemize}

It can be directly checked that the previous map $\partial$ satisfies $\partial^2=0$ (it follows from the relations in the quiver algebra ${R}_n$) and in addition $\partial$ is of degree $(1,0,0)$. Hence we have the following lemma.

\begin{lem} For all admissible curves $\check{c}$ in normal form, $(L(\check{c}), \partial)$ is a trigraded differential module.
\end{lem}

\begin{rem}

An alternative way, using so-called folded diagrams, of presenting this complex of projective modules is given in \cite{KhS}.
\end{rem}

There is a free $\mathbb{Z}^3$-action on trigraded curves and also on the category ${\cat}_n$ (by shifts). The next lemma relates these two actions and can be directly checked from the construction of the differential module $L(\check{c})$.

\begin{lem}\label{Z3action}
For any triple $(r_1,r_2,r_3)$ of integers and any admissible trigraded curve $\check{c}$ we have :
$$L(\chi(r_1,r_2,r_3)\check{c})\cong L(\check{c})[-r_1-nr_3]\{r_2-nr_3\}\left<-r_3\right>.$$
\end{lem}

\begin{rem}
\

\begin{itemize}
\item Notice that the minus sign in the first shift is here to match the standard convention for shifts on cohomology theories.
\item The previous lemma holds when one replaces admissible trigraded curves by $k$-strings.
\end{itemize}
\end{rem}

The aim of the next theorem is to relate the action by endofunctor of the extended affine type $A$ braid group  on $L(\check{c})$ and the complex associated to the image of the curve $\check{c}$ under the mapping class group action. This is done in order to be able to proceed to the Hom-space dimension computations of the next subsection. In addition, let us mention that the proof goes roughly as follows: the image of a complex $L(\check{c})$ under a composition of functors $\Fu_i$ is in general more complicated than the element associated to the image of $\check{c}$ by the element of the mapping class group corresponding to the sequence of functors, and the general procedure is to reduce the first one to the second one (by a sequence of deformation retracts and isomorphisms). 
\begin{thm}{\label{thm2action}}
For any admissible trigraded curve $\check{c}$, we have the following isomorphisms in ${\cat}'_n$:
$$ \Fu_i(L(\check{c})) \cong L(\check{t}_{b_i}(\check{c})) \mbox{ for all } 1\leq i\leq n.$$ and 
$$ \Fu_{\rho}(L(\check{c})) \cong L(\check{t}_{\partial}(\check{c})).$$
\end{thm}
\begin{proof}
For the first case, the proof is exactly similar to the one given by Khovanov and Seidel \cite{KhS}, if $n>3$. In the case $n=3$, one needs to check an additional case in our version of Lemma 4.5 in \cite{KhS}, which is the one where two different $k$-strings have endpoints connected by an essential segment of type $1$ or $1'$. One also needs to take carefully care of the cases $i=1$ and $i=n$, where the third grading comes into play. The second isomorphism is completely immediate and follows from the definition of $\Fu_{\rho}$ and its action on the projectives $P_k$.
\end{proof}

\begin{cor}
For any admissible trigraded curve $\check{c}$ and any element $\sigma \in \hat{\B}_{\hat{A}_{n-1}}$ we have :
$$ \Fu_{\sigma}(L(\check{c})) \cong L(\check{\sigma}(\check{c})).$$
\end{cor}

\subsection{Graded dimensions of Hom-spaces and faithfulness of the categorical action}

The aim of this subsection is to prove that the category ${\cat}'_n$ encodes the intersection numbers. It will be proved by showing that the Poincar\'e polynomial of the space of morphisms between two objects $L(\check{c})$  and $L(\check{c}')$ is equal to the trigraded intersection number between $\check{c}$ and $\check{c}'$.


\begin{lem}
For any $k=1,\dots,n$ and for any trigraded $k$-string $\check{g}$, the abelian group
$$Hom_{{\cat}'_n}({P}_k,L(\check{g})[s_1]\{-s_2\}\left<-s_3\right>)$$
is free for all $(s_1,s_2,s_3)\in\mathbb{Z}^3$. The Poincar\'e polynomial
$$\sum_{(s_1,s_2,s_3)\in \mathbb{Z}^3}\mbox{rk}\Hom_{{\cat}'_n}({P}_k,L(\check{g})[s_1]\{-s_2\}\left<-s_3\right>)q_1^{s_1}q_2^{s_2}q_3^{s_3}$$
is equal to the trigraded intersection number $I^{\tri} (\check{b_k}, \check{g})$.
\end{lem}
\begin{proof}
The proof follows the same line as the one given by Khovanov-Seidel and is in three steps. First, Lemma \ref{Z3action} and property (T3) of the trigraded intersection numbers imply that one can restrict to the case of $k$-strings $\check{g}$ whose left or right endpoint is of degree $(0,0,0)$. Secondly, Theorem \ref{thm2action}  and property (T2) of trigraded intersection numbers imply that one can restrict, for the cases $I$, $II$ and $III$, to the $\check{g}$-strings depicted in Figure \ref{fig:kstringk=2n-1}. Thirdly, a direct computation of the Poincar\'e polynomial in the six different cases of Figure \ref{fig:kstringk=2n-1}, Figure \ref{fig:kstringk=1} and Figure \ref{fig:kstringk=n} and a comparison with the trigraded intersection numbers computed in Lemma \ref{Lemma 3.18} end the proof.
\end{proof}

\begin{lem}
For any trigraded admissible curve in normal form $\check{c}$  and any $k=1,\ldots,n$, we have the following isomorphism:
$$Hom_{{\cat}'_n}({P}_k,L(\check{c})[s_1]\{-s_2\}\left<-s_3\right>)\cong\bigoplus_{\check{g}\in st(\check{c},k)}Hom_{{\cat}'_n}({P}_k,L(\check{g})[s_1]\{-s_2\}\left<-s_3\right>), $$
for all $(s_1,s_2,s_3)\in \mathbb{Z}^3$.
\end{lem}

\begin{proof}
The result follows from two facts: the first is that the space of morphisms between ${P_i}$ and ${P_j}[s_1]\{-s_2\}\left<-s_3\right>$ is trivial if $|i-j|>1$ . The second that any morphism on the right can be extended to a morphism on the left. The latter property follows from the definition of the complex $L(\check{c})$ and the fact that paths containing a subpath of the form $(i-1|i|i+1)$ or $(i+1|i|i-1)$ are zero in the algebra $R_n$. 
\end{proof}

Similarly it follows immediately from the local properties of the trigraded intersection numbers that 
$$I^{\tri} (\check{b_k}, \check{c})=\sum_{\check{g}\in st(\check{c},k)}I^{\tri} (\check{b_k}, \check{g}),$$
see Lemma \ref{Lemma 3.20}. In addition, since the categorical action respects by definition the space of morphisms and similarly the extended affine braid group action respects the trigraded intersection numbers (property (T2)), we have the following proposition:

\begin{prop} For any $\tau$ and $\sigma$ in $\hat{\B}_{\hat{A}_{n-1}}$, and any $k,l=1,\ldots,n$ we have that the Poincar\'e polynomial of 
$$Hom_{{\cat}'_n}(\Fu_{\tau}({P}_k), \Fu_{\sigma}({P}_l))$$
is equal to the trigraded intersection number $I^{\tri}(\check{\tau}(\check{b_k}),\check{\sigma}(\check{b_l}))$.
\end{prop}

We now state the main theorem of this section.

\begin{thm}
If $\Fu_{\sigma}$ acts on the category ${\cat}_n$ as the identity functor, then $\sigma$ is the unit of $\hat{\B}_{\hat{A}_{n-1}}$.
\end{thm}
\begin{proof}
If $\Fu_{\sigma}$ acts on the category ${\cat}_n$ as the identity functor, so does it on ${\cat}'_n$, then it follows from the previous proposition that $\check{\sigma}$ preserves the trigraded intersection numbers of admissible curves. It implies, together with property $(T1)$ of Lemma~\ref{proprieteT}, that, for all $k,l=1,\ldots,n$, we have $I(b_k,\sigma(b_l))=I(b_k,b_l)$. Then, by Lemma \ref{geomintnb2}, one knows that  $\sigma$ is hence of the form $\rho^{np}$ for some integer $p$. Now one can immediately check that $\Fu_{\rho^{np}}$ is then the shift functor $<-p>$, which acts as the identity functor if and only if $p=0$ and this finishes the proof.
\end{proof}

\paragraph*{\bf Acknowledgements}
\ 
\newline
A.G. was partially supported by CMUP (UID/MAT/00144/2013), which is funded by FCT (Portugal) with national (MEC) and European structural funds through the programs FEDER, under the partnership agreement PT2020.  A.-L.T. is grateful to Institut Mittag-Leffler for the warm hospitality and the support during the final stages of this project. She thanks also warmly Volodymyr Mazorchuk for useful discussions about this project. E.W is partially supported by the French ANR research project VasKho ANR-11-JS01-00201. He thanks Christian Blanchet and Vincent Florens for many discussions of topological nature. The authors acknowledge IMB, CMUP and Uppsala Universitet for providing nice working enviromnents that have made the completion of this project possible.

\bibliographystyle{alpha}
\bibliography{biblioMC}
\vspace{0.1in}

\noindent A.G.: { \sl \small Centro de Matem\'atica da Universidade do Porto
Departamentos de Matem\'atica, 4169-007 Porto, Portugal} 
\newline \noindent {\tt \small email: agnes.gadbled@fc.up.pt}

\noindent A.-L.T.: { \sl \small Matematiska Institutionen, Uppsala Universitet, 75106 Uppsala, Sverige; Mittag-Leffler Institute, 18260 Djursholm, Sverige} 
\newline \noindent {\tt \small email: anne-laure.thiel@math.uu.se}

\noindent E.W.: { \sl \small Institut Math\'ematique de Bourgogne, UMR 5584, Universit\'e de Bourgogne, 21078 Dijon, France} 
\newline \noindent {\tt \small email: emmanuel.wagner@u-bourgogne.fr}

\end{document}